\documentclass{article}

\usepackage[a4paper, left=3.3cm,top=3.3cm,right=3.3cm,bottom=3.3cm]{geometry}
\usepackage{concmath}
\usepackage[T1]{fontenc}

\usepackage{cite}
\usepackage{amsmath,amssymb,amsfonts,amsthm}
\usepackage{algorithmic}
\usepackage{graphicx}
\usepackage{textcomp}
\usepackage{xcolor}
\usepackage{enumitem}

\usepackage{authblk}
\usepackage{todonotes}
 \usepackage{psfrag}
  \usepackage{graphics}
   \usepackage{hyperref}
\usepackage{siunitx}
\usepackage{booktabs}

\usepackage{soul}
\usepackage{xcolor}
\colorlet{lred}{red!40}
\colorlet{lgreen}{green!40}
\colorlet{lblue}{blue!40}

\newcommand{\rr}{\mathbf r}
\newcommand{\rss}{\mathbf s}
\newcommand{\Lsource}{{\boldsymbol{h}}}
\newcommand{\R}{\mathbb R}

\theoremstyle{theorem}
\newtheorem{theorem}{Theorem}

\newcommand{\trans}{\mathsf{T}}

\newcommand{\sign}{\mathrm{sign}}

\newcommand{\Fullop}{\mathcal A}
\newcommand{\Id}{\operatorname{Id}}

\newcommand{\Wave}{\mathcal  W}

\newcommand{\samp}{\mathbf S}
\newcommand{\Samp}{\mathcal S}

\newcommand{\fv}{\boldsymbol{f}}

\newcommand{\DDelta}{{\mathcal{L}}}
\newcommand{\ppartial}{{\mathcal{D}}}

\newcommand{\Eeta}{{\boldsymbol{\eta}}}
\newcommand{\Source}{{\boldsymbol{f}}}
\newcommand{\Data}{\boldsymbol{g}}
\newcommand{\hh}{\boldsymbol{z}}

\newcommand{\source}{p_0}

\newcommand{\UU}{\mathcal U}
\newcommand{\NN}{\mathcal N}
\newcommand{\BP}{\mathcal B}
\newcommand{\rec}{\mathcal R}

\newcommand{\noise}{\eps}

\def\BibTeX{{\rm B\kern-.05em{\sc i\kern-.025em b}\kern-.08em
    T\kern-.1667em\lower.7ex\hbox{E}\kern-.125emX}}

\newcommand{\hilbert}[1]{\mathcal{H}}

\newcommand{\edot}{\,\cdot\,}

\newcommand{\Po}{\mathcal P}

\newcommand{\Io}{\mathbf I}

\newcommand{\Kern}{\operatorname{Ker}}

\DeclareMathOperator*{\argmin}{arg\,min}

\newcommand{\supp}{\operatorname{supp}}

\newcommand{\rmd}{\mathrm d}

\newcommand{\eps}{\epsilon}

\newcommand{\al}{\alpha}

\newcommand{\prox}{\operatorname{prox}}

\newcommand\sset[1]{\{#1\}}
\newcommand\abs[1]{\left\vert#1\right\vert}
\newcommand\set[1]{\left\{#1\right\}}
\newcommand\sabs[1]{\lvert#1\rvert}
\newcommand\norm[1]{\Vert#1\Vert}
\newcommand\snorm[1]{\Vert#1\Vert}

\newcommand{\kl}[1]{\left(#1\right)}

\newcommand{\Xin}{\boldsymbol b}
\newcommand{\Xout}{\boldsymbol f}

\usepackage{soul}
\usepackage{xcolor}

\colorlet{lred}{red!40}
\colorlet{lgreen}{green!40}
\colorlet{lblue}{blue!40}
\definecolor{mixc}{cmyk}{0.5,0.5,0.5,0}
\colorlet{mixl}{mixc!30}

\numberwithin{equation}{section}
\numberwithin{theorem}{section}
\numberwithin{figure}{section}

\author{Stephan Antholzer}

\affil{Department of Mathematics, University of Innsbruck\authorcr
Technikerstrasse 13, 6020 Innsbruck, Austria\authorcr
{\tt  stephan.antholzer@uibk.ac.at}}

\author{Johannes Schwab}

\affil{Department of Mathematics, University of Innsbruck\authorcr
Technikerstrasse 13, 6020 Innsbruck, Austria\authorcr
{\tt johannes.schwab@uibk.ac.at}}

\author{Markus Haltmeier}

\affil{Department of Mathematics, University of Innsbruck\authorcr
Technikerstrasse 13, 6020 Innsbruck, Austria\authorcr
E-mail: {\tt markus.haltmeier@uibk.ac.at}}

\date{Januar 30, 2018}

\begin{document}

\title{Deep learning versus $\ell^1$-minimization for compressed sensing photoacoustic tomography}

\maketitle

\begin{abstract}

We investigate compressed sensing (CS) techniques for reducing the number of measurements in photoacoustic tomography (PAT). High resolution imaging from CS data requires particular image reconstruction algorithms.  The most established reconstruction techniques for that purpose use sparsity and $\ell^1$-minimization.  Recently, deep learning appeared as a new paradigm for CS and other inverse problems.    In this paper, we compare a recently invented joint $\ell^1$-minimization algorithm with two deep learning methods, namely a residual network and an approximate nullspace network. We present numerical results showing that all developed techniques   perform well for deterministic sparse measurements as well as for random Bernoulli measurements. For the deterministic sampling, deep learning shows more accurate results, whereas for Bernoulli measurements the $\ell^1$-minimization algorithm performs best. Comparing the implemented deep learning approaches, we show that the nullspace network uniformly outperforms the residual network in terms of the mean squared error (MSE).


\end{abstract}

\section{Introduction}

Compressed sensing (CS)  allows to reduce the number of
measurements in photoacoustic  tomography (PAT) while preserving  high spatial resolution.
A reduced number of  measurements can increase the measurement  speed and reduce system
costs  \cite{haltmeier2018sparsification,sandbichler2015novel,haltmeier2016compressed,betcke2016acoustic,provost2009application} .
However, CS PAT  image reconstruction requires special algorithms to achive high resolution
imaging. In this work, we compare $\ell^1$-minimization  and deep learning algorithms for
 2D PAT. Among others, the two-dimensional case arises in PAT with integrating line detectors \cite{paltauf2007photacoustic,burgholzer2007temporal}.

In the case that a sufficiently  large number of detectors is used, according to Shannon's sampling theory,
implementations of full data methods yield almost artifact free reconstructions  \cite{haltmeier2016sampling}. As the fabrication of an array of detectors is demanding,
experiments using integrating line detectors are often carried out using a single line detector, scanned on circular paths using scanning stages~\cite{NusEtAl10, GruEtAl10}, which is very time consuming.  Recently, systems using arrays of $64$ parallel line detectors have been demonstrated~\cite{gratt201564line, beuermarschallinger2015photacoustic}. To keep production costs low and to allow fast imaging, the number of measurements  will typically be kept much smaller than advised by Shannon's sampling theory and one has to deal with highly under-sampled data.

After discretization, image reconstruction in CS PAT consists in solving the inverse
problem
\begin{equation} \label{eq:ip}
\Data =   \Fullop  \fv  + \noise   \,,
 \end{equation}
where  $\fv\in \R^{n}$ is the discrete photoacoustic (PA)   source  to be reconstructed,
$\Data  \in \R^{mQ}$ are the given CS data, $\noise$  is the noise in the
data and $\Fullop \colon \R^n \to \R^{mQ} $ is the forward matrix.  The forward matrix is the
product of the PAT full data problem and the compressed sensing measurement matrix.
Making CS measurements in PAT implies that  $mQ \ll n$  and  therefore,
even in the  case of exact data, solving  \eqref{eq:ip} requires  particular
reconstruction algorithms.

\subsection{CS PAT recovery algorithms}

Standard  CS reconstruction techniques  for \eqref{eq:ip}
are based on sparse recovery via $\ell^1$-minimization.
These algorithms rely on sparsity of the unknowns
in a suitable basis or dictionary and special  incoherence of the  forward matrix.
See \cite{sandbichler2015novel,haltmeier2016compressed,betcke2016acoustic,provost2009application} for different CS approaches in PAT.
To guarantee sparsity of the unknowns, in
\cite{haltmeier2018sparsification} a new sparsification and corresponding joint
$\ell^1$-minimization   have been derived.
Recently, deep learning  appeared as a new  reconstruction paradigm
for CS  and other inverse problems.  
Deep learning approaches for PAT  can be found in
\cite{antholzer2018deep,antholzer2018photoacoustic,kelly2017deep,allman2018photoacoustic,schwab2018fast,hauptmann2018model,waibel2018reconstruction}.

In this work, we compare the performance of the joint  $\ell^1$-minimization  algorithm of \cite{haltmeier2018sparsification} with deep learning
approaches for CS PAT image reconstruction. For the latter we use the  residual network  \cite{jin2016deep,han2016deep,antholzer2018deep}
and the nullspace network \cite{mardani2017deep,schwab2018deep}. The nullspace network
includes a certain data consistency layer and even has been shown to be a
regularization method in~\cite{schwab2018deep}. Our results show that  the
nullspace network uniformly outperforms the residual network for CS PAT in terms of the mean squared error (MSE).

\subsection{Outline}

In Section~\ref{sec:cspat}, we present  the required background
from  CS PAT. The sparsification strategy and the joint  $\ell^1$-minimization algorithm  are summarized in Section \ref{sec:sparse}.
The employed  deep learning  image reconstruction strategies using a residual network
and  an (approximate) nullspace network  are described in  Section~\ref{sec:deep}.
In Section~\ref{sec:num} we present reconstruction results for sparse
 measurements and Bernoulli measurements.
 The paper ends with a discussion in  Section~\ref{sec:conclusion}.

\begin{psfrags}
\begin{figure}[htb!]
\begin{center}
  \includegraphics[width=1\columnwidth]{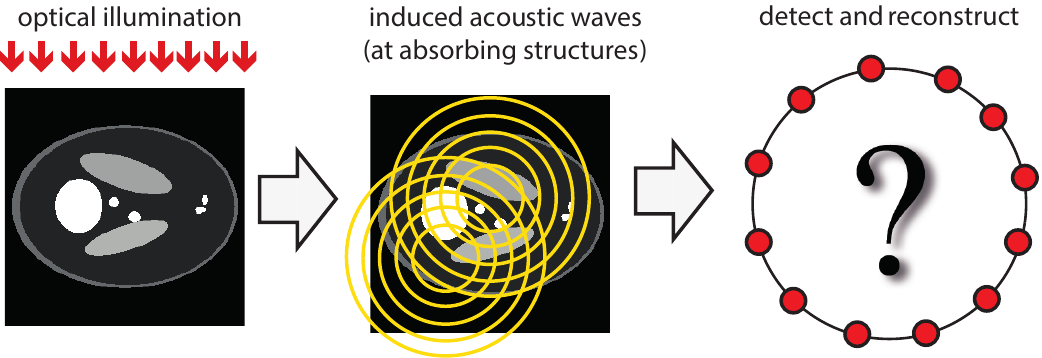}
\caption{\label{fig:pat} (a) An object is illuminated with a  short optical pulse;  (b) the absorbed light distribution causes an acoustic pressure; (c) the  acoustic pressure is measured outside the object and used to reconstruct an image of the interior.}
\end{center}
\end{figure}
\end{psfrags}

\section{Compressed photoacoustic tomography}

\label{sec:cspat}

\subsection{Photoacoustic tomography}

As illustrated in Figure~\ref{fig:pat}, PAT is based
on generating an acoustic wave inside some investigated object using short optical pulses.
Let us denote by $\source \colon \R^d \to \R$ the initial pressure
distribution which provides diagnostic information about the
patient and which is the quantity of interest  in  PAT \cite{kuchment2011mathematics,paltauf2007photacoustic,wang2006photoacoustic}.
For keeping the presentation simple and focusing on  the main ideas we only consider the  case of  $d=2$.
 Among others, the two-dimensional case arises in PAT with so called integrating line detectors \cite{paltauf2007photacoustic,burgholzer2007temporal}.
Further, we restrict ourselves to the case of a circular measurement geometry, where the acoustic measurements are made on a circle  surrounding the investigated object.

In two spatial dimensions, the induced  pressure in PAT satisfies the   2D wave equation
\begin{multline} \label{eq:wave}
\partial^2_t p (\rr,t)  - c^2 \Delta p(\rr,t) \\
= \delta'(t)  \source (\rr)  \quad \text{ for } (\rr,t) \in \R^2 \times \R_+ \,.
\end{multline}
Here  $\rr  \in \R^2$ is the spatial location, $t \in \R$ the time variable, $\Delta_{\rr}$ the spatial Laplacian, $c$ the speed of sound, and $\source (\rr)$ the  PA source that is  assumed to vanish
outside  the  disc $B_R \triangleq  \sset{x \in \R^2 \mid  \norm{x} < R}$ and  has to be recovered.
The wave equation \eqref{eq:wave} is augmented with
$p(\rr, t) =0$  on $\set{t < 0}$. The acoustic pressure is then uniquely defined and referred to as the causal solution of~\eqref{eq:wave}.

PAT in a circular measurement geometry consist in recovering the function $\source$ from measurements   of $p(\rss,t) $ on $\partial  B_R  \times  (0,\infty)$. In the case of full data, exact  and  stable PAT image reconstruction is  possible  \cite{haltmeier2017iterative,stefanov2009thermoacoustic} and several efficient methods for recovering
 $\Source$   are available. As an example, we mention the  FBP
formula derived in \cite{FinHalRak07},
\begin{equation} \label{eq:fbp2d}
          \source(\rr)
         =
        - \frac{1}{\pi R}
        \int_{\partial B_R}
        \int_{\abs{\rr-z}}^\infty
        \frac{ (\partial_t t p)(\rss, t)}{ \sqrt{t^2-\sabs{\rr-\rss}^2}}  \, \rmd t
        \rmd S(\rss)
         \,.
\end{equation}
Note the inversion operator in \eqref{eq:fbp2d} is also the adjoint of the
forward operator, see \cite{FinHalRak07}.

\subsection{Discretization}

In practical applications, the acoustic pressure  can only be
measured with a finite number of acoustic detectors.
The standard sampling scheme  for PAT in circular geometry  assumes
uniformly sampled  values
\begin{equation} \label{eq:data}
    p \kl{ \rss_k, t_\ell}
    \text{ for }
    ( k, \ell) \in \set{ 1, \dots,  M} \times \set{ 1, \dots,  Q }\,,
\end{equation}
  with
\begin{align}
     \rss_k
	 &\triangleq
     \begin{bmatrix} R \cos \kl{2\pi(k-1)/M} \\ R\sin \kl{2\pi(k-1)/M}  \end{bmatrix}
     \\
     t_\ell
	 &\triangleq
    2R (\ell-1) /(Q-1)
      \,.
\end{align}
The number $M$ of detector positions in  \eqref{eq:data} is directly related to the resolution of the final reconstruction. Namely,
\begin{equation} \label{eq:samplingcondition}
M \geq 2 R_0 \lambda_0
\end{equation}
equally spaced transducers are required  to stably recover any
PA source $\source$  that has  maximal essential wavelength  $\lambda_0$ and is supported in a disc
$B_{R_0} \subseteq  B_{R}$; see \cite{haltmeier2016sampling}.
Image reconstruction in this case can be performed   by discretizing the inversion formula  \eqref{eq:fbp2d}.
The     sampling  condition \eqref{eq:samplingcondition} requires a very high sampling rate,  especially when the PA source contains narrow features, such as blood vessels or sharp interfaces.

Note that temporal samples can easily be collected at a high sampling rate compared to the spatial sampling, where each sample requires a separate sensor.
It is therefore beneficial to keep $M$ as small as possible.
 Consequently, full sampling  in PAT is costly and time consuming and strategies for 
 reducing the number of detector locations are desirable.

\subsection{Compressive  measurements in PAT}

To reduce the number of measurements   we use CS measurements.
Instead of collecting $M$ individually sampled signals as in \eqref{eq:data}, we take general linear measurements
\begin{equation} \label{eq:cs}
	\Data(j, \ell ) \triangleq \sum_{k=1}^M
	\samp[j, k] p(\rr_k, t_\ell )
	\; \text{ for } j \in \set{  1, \dots, m} \,,
\end{equation}
with  $m \ll M$. Several choices for the measurement matrix $\samp$ are possible and have been used for CS PAT
\cite{sandbichler2015novel,haltmeier2016compressed,betcke2016acoustic}.
In this work, we  take $\samp$ as deterministic sparse subsampling  matrix or
Bernoulli random matrix; see Subsection \ref{eq:setup}.

Let us denote by $\Wave  \in  \R^{MQ \times n} $ the  discretized  solution operator
of the wave equation and by  $ \Samp  \triangleq  \samp  \otimes \Io \in \R^{ mQ \times MQ    }$
the Kronecker (or tensor)  product   between the CS measurement matrix  $\samp$
and the identity matrix $\Io$. Then the CS data \eqref{eq:cs} written as column vector
$\Data \in \R^{mQ}$ are given by
\begin{equation} \label{eq:ip0}
\Data =    \Fullop   \fv  \quad \text { with }  \Fullop  \triangleq  \Samp  \circ  \Wave  \in \R^{mQ \times n} \,.
 \end{equation}
In the case of CS measurements we have $mQ \ll n$ and
therefore  \eqref{eq:ip0} is highly underdetermined and image reconstruction
requires special reconstruction algorithms.

\section{Joint $\ell^1$-minimization for CS PAT}
\label{sec:sparse}

Standard  CS  image reconstruction is based on $\ell^1$ minimization and
sparsity of the unknowns to be recovered. In \cite{haltmeier2018sparsification} we introduced
a sparse recovery strategy that we will use in the present paper and recall below.

\subsection{Background from $\ell^1$-minimization}

An element $\Lsource \in \R^n$ is called $s$-sparse if
it contains at most  $s$ nonzero elements. If we are given measurements $\Fullop \Lsource  =  \Data$
where $\Lsource \in \R^n$ and $\Data  \in\R^{mQ}$  with $mQ \ll n$,
then stable recovery of  $\Lsource$  from $\Data$  via  $\ell^1$-minimization
can be guaranteed if $\Lsource$ is sparse and the matrix $\Fullop$ satisfies the
restricted isometry property of order $2s$.
The latter property  means that for all $2s$-sparse vectors $\hh \in \R^n$ we have
\begin{equation} \label{eq:RIP}
(1-\delta) \norm{\hh}^2\leq \norm{ \Fullop \hh}^2 \leq(1+\delta) \norm{\hh}^2 \,,
\end{equation}
for an RIP constant $\delta < 1 / \sqrt{2}$; see~\cite{foucart2013mathematical}.

Bernoulli random matrices satisfy the RIP with
high probability~\cite{baraniuk2008simple} whereas the
subsampling matrix clearly does not satisfy the RIP.
In the case of CS PAT, the forward matrix is given by
$ \Fullop = (\samp  \otimes \Io ) \circ \Wave$.
It is not known whether  $\Fullop$ satisfies the RIP for either
the Bernoulli of the subsampling matrix. In such situations one may use
the following stable  reconstruction result from inverse problems theory.

\begin{theorem}[$\ell^1$-minimization] \label{thm:ell1}
Let $\Fullop \in \R^{mQ \times n}$ and $\Lsource \in \R^{n}$
Assume
\begin{align} \label{eq:ssc-1}
      &\exists \Eeta \in \R^{mQ} \colon \Fullop^\trans  \Eeta \in  \sign(\Lsource)
      \\ \label{eq:ssc-2}
      &\abs{(\Fullop^\trans  \Eeta)_i} <  1
       \text{ for } i \not \in \supp (\Lsource)\,,
      \end{align}
where $\sign(\Lsource)$ is the set valued signum function  and
 $\supp (\Lsource)$ the set of all nonzero entries of $\Lsource$,
and that the restriction of $\Fullop$  to  the subspace  spanned by $e_i$ for
$ i \in \supp (\Lsource)$ is injective.
Then for any  $\Data^\delta  \in \R^{mQ}$ with  $\snorm{ \Fullop \Lsource  -  \Data^\delta}_2 \leq \delta$,
any minimizer of the   $\ell^1$-Tikhonov functional
\begin{equation} \label{eq:tikhonov}
\Lsource_\beta^\delta \in \argmin_{\hh} \frac{1}{2} \norm{ \Fullop \hh  -  \Data^\delta }_2^2 + \beta \norm{\hh}_1
\end{equation}
satisfies $\snorm{ \Lsource_\beta^\delta - \Lsource }_2 = \mathcal{O} (\delta)$
provided  $\beta \asymp \delta$. In particular, $\Lsource$ is the unique
$\norm{\edot}_1$-minimizing solution of $\Fullop \hh= \Data$.
\end{theorem}

 \begin{proof}
 See \cite{Gra11}.
 \end{proof}

 In \cite{candes2005decoding,Gra11} it is shown that the RIP  implies the conditions in
 Theorem~\ref{thm:ell1}.  Moreover, the smaller $\supp (\Lsource)$, the easier the
 conditions in Theorems are  satisfied.
 Therefore, sufficient sparsity of the unknowns is a crucial condition
 for the success of $\ell^1$-minimization.

\subsection{Sparsification strategy}

The used CSPAT approach in \cite{haltmeier2018sparsification}
is based on following theorem which allows bringing
sparsity into play.

\begin{theorem} \label{thm:laplace}
Let  $\source$  be a given   PA source  vanishing outside $B_R$,
and let  $p$ denote the causal solution of~\eqref{eq:wave}.
Then  $\partial_t^2 p$  is the causal solution of
\begin{multline} \label{eq:wavesparse}
\partial^2_t q (\rr,t)  - c^2 \Delta q(\rr,t) \\
= \delta'(t)   c^2 \Delta \Source (\rr)  \quad \text{ for } (\rr,t) \in \R^2 \times \R_+ \,.
\end{multline}
In particular, up to  discretization error, we have
\begin{equation}
\forall \Source \in \R^n \colon \quad \ppartial_t^2 \Fullop  [\Source] = \Fullop  [c^2 \DDelta_\rr \Source] \,,
\end{equation}
where $\Fullop = (\samp  \otimes \Io ) \circ \Wave$ denotes the discrete
CS PAT forward  operator defined by \eqref{eq:ip0}, $\DDelta_\rr$ is the discretized
Laplacian,  and $\ppartial_t$ the discretized temporal derivate.
\end{theorem}

\begin{proof}
See \cite{haltmeier2018sparsification}.
\end{proof}

Typical  phantoms  consist of smoothly varying parts
and rapid changes  at interfaces. For such PA sources,
the modified source  $c^2 \DDelta_{\rr} \Source$
is sparse or at least compressible. The theory of
CS therefore predicts that the modified source
can be recovered by solving via $\ell^1$-minimization
 \begin{equation} \label{eq:L1exact}
\min_{ \Lsource}     \norm{\Lsource}_1
\quad \text{such that }     \Fullop \Lsource  =   \ppartial_t^2 \Data       \,.
\end{equation}
Having obtained an approximate minimizer  $\Lsource$ by either solving
\eqref{eq:L1exact} or  its relaxed version,  one can recover the original
PA source $\Source $ by subsequently solving the Poisson
equation $   \DDelta_{\rr} \Source  = \Lsource/c^2$ with zero
boundary conditions. Using  the above two-stage procedure,
we observed disturbing low frequency artifacts in the reconstruction.
Therefore, in \cite{haltmeier2018sparsification}  we
introduced a  different  joint  $\ell^1$-minimization approach based on
Theorem~\ref{thm:laplace} that jointly recovers $\Source$ and $ c^2 \DDelta_{\rr} \Source$.

\subsection{Joint $\ell^1$-minimization framework}

The modified data $\ppartial_t^2 \Data$ is  well suited
to recover  singularities of   $\Source$, but hardly contains low-frequency
components of $\Source$. On the other hand, the low frequency
information is contained in the original data, which is still available to us.
This motivates the following joint $\ell^1$-minimization problem
\begin{equation} \label{eq:joint2}
\begin{aligned}
&\min_{(\Source, \Lsource)}  \norm{\Lsource}_1  +  I_{C} (\Source)  \\
&\text{such that }
\begin{bmatrix}  \Fullop\Source,  \Fullop\Lsource ,  \DDelta_{\rr} \Source  -    \Lsource/c^{2} \end{bmatrix} =
\begin{bmatrix} \Data , \ppartial_t^2 \Data ,0 \end{bmatrix}   \,.
\end{aligned}
\end{equation}
Here $I_C $  is the indicator function of
$ C  \triangleq [0, \infty)^n$, defined by
$I_C(\Source) = 0 $ if  $\Source \in C$ and $I_C(\Source) =\infty $ otherwise, and
guarantees non-negativity.

\begin{theorem} \label{thm:recovery}
Assume that $\Source \in \R^n$  is non-negative,
that   the measurement matrix $\Fullop$ and the modified PA source
$\Lsource =  c^2 \DDelta_{\rr} \Source$ satisfy  Equations \eqref{eq:ssc-1}, \eqref{eq:ssc-2},
and denote $\Data = \Fullop\Source$.
Then, the pair $[\Source, c^2 \DDelta_{\rr} \Source]$ can be recovered as the
unique solution of the joint  $\ell^1$-minimization problem~\eqref{eq:joint2}.
\end{theorem}

\begin{proof}
See \cite{haltmeier2018sparsification}.
\end{proof}

In the case the data is only approximately sparse or noisy, we propose, instead of
\eqref{eq:joint2}, to solve the $\ell^2$-relaxed version
\begin{multline} \label{eq:joint-pen}
\frac{1}{2} \norm{\Fullop \Source- \Data}_2^2
+
\frac{1}{2} \norm{\Fullop\Lsource - \ppartial_t^2 \Data}_2^2
+
\frac{\alpha}{2} \norm{\DDelta_{\rr} \Source  -    \Lsource/c^{2}}_2^2
\\ +
 \beta \norm{\Lsource}_1  +  I_{C} (\Source)
\to \min_{(\Source, \Lsource)} \,.
\end{multline}
Here $\alpha>0$ is a tuning and  $\beta >0$ a regularization
parameter.

\subsection{Numerical minimization}

We will solve~\eqref{eq:joint-pen} using a proximal forward-backward splitting method~\cite{combettes2011proximal}, which is well suited for minimizing
the sum  of a smooth and a non-smooth but convex part.
In the case of  \eqref{eq:joint-pen}  we take the smooth part as
\begin{multline}
\Phi(\Source,\Lsource) \triangleq  \frac{1}{2} \norm{\Fullop \Source- \Data}_2^2
 \\ +
\frac{1}{2} \norm{\Fullop\Lsource - \ppartial_t^2\Data}_2^2
 +
\frac{\alpha}{2} \norm{\DDelta_{\rr} \Source  -    \Lsource/c^{2}}_2^2
\end{multline}
and the  non-smooth part as $\Psi(f,h) \triangleq \beta \norm{\Lsource}_1  +  I_{C}(f)$.

The proximal gradient  algorithm then alternately  performs an explicit gradient step
for $\Phi$ and an implicit proximal step for $\Psi$. For the proximal step, the proximity operator
of a function must be computed. The proximity operator of a given convex function $F\colon \R^n \to\R$ is defined by~\cite{combettes2011proximal}
	\[\prox_{F}(\Source) \triangleq
	\operatorname{argmin}
	\set{ F(\hh)+\tfrac{1}{2} \| \Source - \hh \|_2^2 \mid \hh \in\R^n }  \,.\]
The regularizers we are considering here have the advantage, that their proximity operators can be computed explicitly and do not cause a significant computational overhead.
The gradient $[\nabla_\Source \Phi, \nabla_\Lsource \Phi]$ of the  smooth part can
easily be computed  to be
\begin{align*}
\nabla_\Source \Phi (\Source,\Lsource)
&=  \Fullop^* (\Fullop \Source- \Data)-  \alpha \DDelta_{\rr} (\DDelta_{\rr} \Source  -    \Lsource/c^{2})\\
\nabla_\Lsource \Phi (\Source,\Lsource)
&=   \Fullop^* (\Fullop \Lsource- \ppartial_t^2\Data)- \frac{\alpha }{c^2}(\DDelta_{\rr} \Source  -    \Lsource/c^{2}) \,.
\end{align*}
The proximal operator of the non-smooth part is given by
\begin{align*}
\prox(\Source,\Lsource) &:= [\prox_{I_C}(\Source), \prox_{\beta\|\cdot\|_1(\Lsource)}] \,,\\
\prox_{I_C}(\Source)_i &= (\max(\Source_i,0))_i \,,\\
\prox_{\beta\|\cdot\|_1}(\Lsource)_i  &= (\max(|\Lsource_i|-\beta,0)\,\sign(\Lsource_i))_i
\end{align*}
With this, the proximal gradient algorithm is given by
\begin{align} \label{eq:prox1}
\Source^{k+1} &= \prox_{I_C}\left(\Source^k - \mu  \nabla_\Source  \Phi (\Source^k, \Lsource^k)  \right)
\\ \label{eq:prox2}
\Lsource^{k+1} &= \prox_{\mu\beta\|\cdot\|_1}\left(\Lsource^k - \mu  \nabla_\Lsource \Phi (\Source^k, \Lsource^k)\right),
\end{align}
where $(\Source^k, \Lsource^k)$ is the $k$-th iterate and $\mu$ the step size.
We initialize the proximal gradient algorithm with $\Source^0=\Lsource^0=0$.

\section{Deep learning for CS PAT}
\label{sec:deep}

As an alternative to the joint $\ell^1$-minimization algorithm
we use deep learning  or CS image reconstruction. 
We thereby use a trained residual  network as well as a corresponding 
(approximate) nullspace network, which offers improved data consistence.

\subsection{Image reconstruction by deep learning}
\label{sec:nn}

Deep learning is a recent paradigm to solve inverse problems of  the form
\eqref{eq:ip}.  In this case, image reconstruction is  performed by an explicit reconstruction
function
\begin{equation} \label{eq:nn}
	 \rec_{\theta}  =  \NN_\theta  \circ  \Fullop^\sharp \colon \R^{mQ} \to \R^n \,.
\end{equation}
The reconstruction operator $\rec_{\theta}$ is the  composition of a  backprojection
operator and  a convolutional neural network
\begin{align}
	   \Fullop^\sharp \colon \R^{mQ} \to \R^n \\
	   \NN_\theta \colon \R^{n} \to \R^n \,.
\end{align}
The  backprojection $\Fullop^\sharp$ performs an  initial reconstruction that is
subsequently  improved by the CNN $\NN_\theta$.
In this work, we use the filtered backprojection  (FBP) algorithm \cite{FinHalRak07}  for
$\Fullop^\sharp$,  which is a discretization of the inversion formula \eqref{eq:fbp2d}.
For the CNN $\NN_\theta$ we  use the  residual network
(see Subsection~\ref{sec:resnet})   and the nullspace network (see Subsection~\ref{sec:nullnet}).

The CNN   is taken from a parameterized family, where parameterization
$ \theta \in \Theta \mapsto \NN_\theta $ is determined by the network architecture.
For adjusting  the parameters, one assumes a family of training
data $ ((\Xin_k, \Xout_k))_{k=1}^N$ is  given where any
training example consist of artifact-free output image $\Xout_k$
and a corresponding input image $ \Xin_k = \Fullop^\sharp \Fullop (\Xout_k) $. The free parameters
$\theta$  are chosen in such a way, that the overall error
of the network for predicting $\Xout_k$ from  $\Xin_k$ is minimized.
The minimization procedure used  in this paper is described in
Subsection \eqref{sec:training}.

\subsection{Residual network}
\label{sec:resnet}

The architecture of the CNN is a  crucial step
for the performance of tomographic image reconstruction  with deep learning.
A  common  architecture in that context  is the following residual network
\begin{equation}\label{eq:resnet}
\rec_{\theta}^{\rm res}  =  (\Id + \UU_\theta) \Fullop^\sharp  \,,
\end{equation}
where $ \UU_\theta$  is the Unet, originally   introduced in
\cite{ronneberger2015unet} for biomedical image segmentation. The
residual network \ref{eq:resnet} has successfully  been  used for various tomographic
image reconstruction tasks \cite{antholzer2018deep,jin2016deep,han2016deep}
including PAT.

\begin{figure}[htb!]
\centering
   \includegraphics[width=\columnwidth]{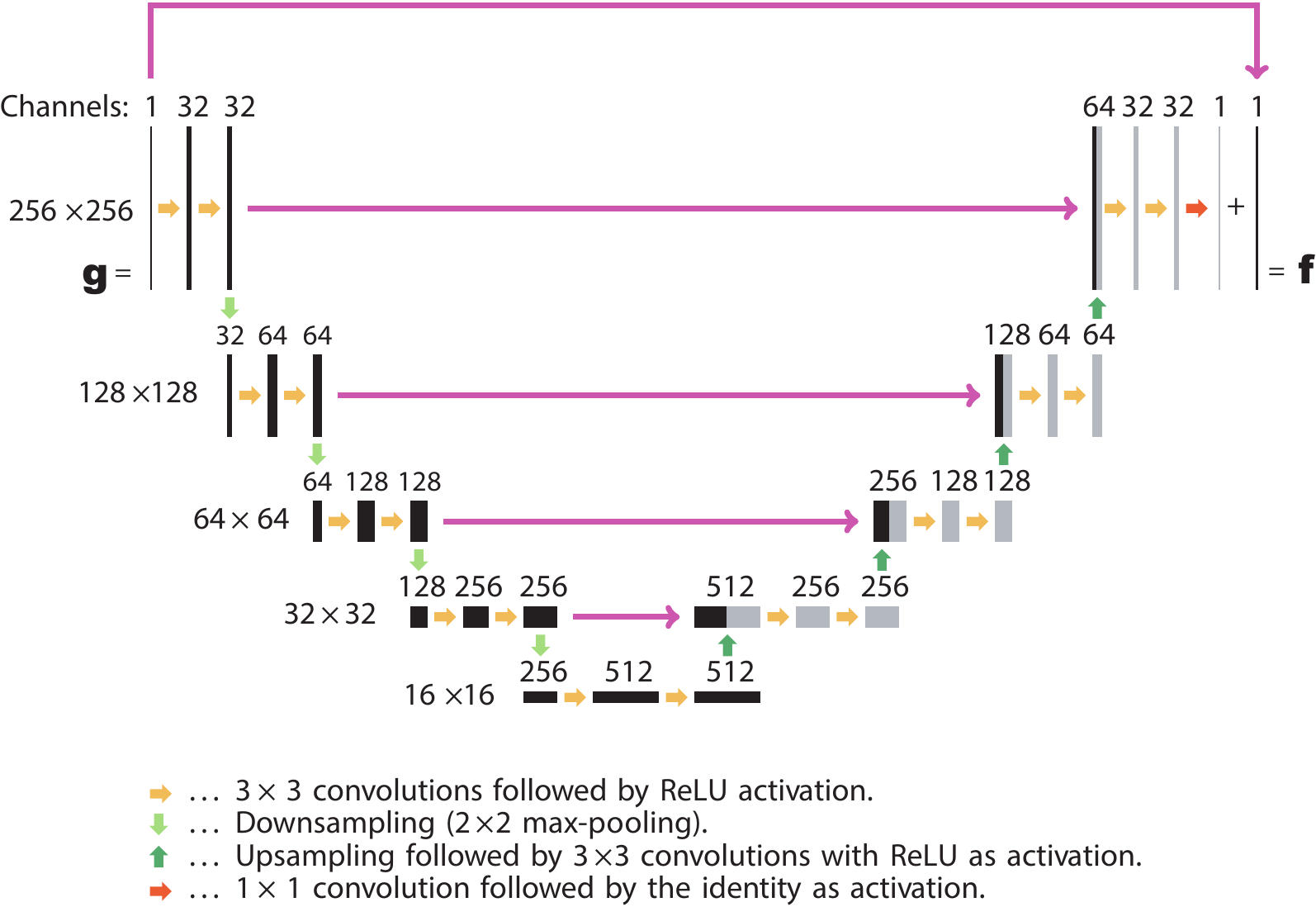}
\caption{ Architecture of the residual network $\Id + \UU_\theta$. \label{fig:net}
The number written above each layer denotes the number of convolution kernels (channels).
The numbers written on the  left are the image sizes.
The long  arrows indicate direct connections with subsequent concatenation or summation.}
\end{figure}

Using  $\Id + \UU_\theta$ instead of  $ \UU_\theta$  affects that actually
the residual images  $\Xout  + \Xin  $ are  learned by the Unet. The residual images
often have a simpler structure  than the original outputs $\Xout $.
As argued in~\cite{han2016deep}, learning the residuals and adding them to the inputs after the last layer is more effective than directly training  for the outputs.
The resulting deep neural network architecture is shown in Figure~\ref{fig:net}.

\subsection{Nullspace network}
\label{sec:nullnet}

Especially when applying $\rec_{\theta}^{\rm res}$ to objects very different from the training set,
the residual network  \eqref{eq:resnet} lacks data consistency, in the sense that
$\rec_{\theta}^{\rm res} \Data$ is  not necessarily   a solution of the given equation
$\Fullop \Source = \Data$.
To overcome this limitation, as an alternative we use the nullspace network
\cite{schwab2018deep},
\begin{equation}\label{eq:nullnet}
    \rec_{\theta}^{\rm null}  =  (\Id + \Po_{ \Kern (\Fullop)} \UU_\theta ) \Fullop^\sharp   \,.
\end{equation}
One strength  of the nullspace network is that the term
 $\Po_{ \Kern (\Fullop)} \UU_\theta$ only adds information  that is
 consistent with the given data.  For example, if $\Fullop^\sharp = \Fullop^+$ equals  the
 pseudoinverse, then $\rec_{\theta}^{\rm null} \Data$  even
 is fully data consistent as implied by the following theorem.

\begin{theorem} \label{thm:null}
Let  $\Data =  \Fullop(\Source^\star)$ be in the range of the forward operator, write  $L(\Fullop, \Data)$  for the set of solutions of the equation
$\Fullop \Source =\Data$ and take $\Fullop^\sharp = \Fullop^+$ as the
 pseudoinverse.
\begin{enumerate}
\item  $ \rec_{\theta}^{\rm null} (\Data) $ is a solution of  $\Fullop \Source = \Data$.
\item We have $ \rec_{\theta}^{\rm null}(\Data)  =  \Po_{ L(\Fullop, \Data)} \rec_{\theta}^{\rm res}(\Data)$.
\item Consider the iteration
\begin{align} \label{eq:iter1}
\Source^{(0)}  &=  \rec_{\theta}^{\rm res}(\Data)
\\ \label{eq:iter2}
\Source^{(k+1)}  &=  \Source^{(k)} - s\Fullop^\trans ( \Fullop  \Source^{(k)} - \Data ) \,,
\end{align}
with step size  $0 <  s  < \snorm{\Fullop}^{-2}$.
Then:
\begin{enumerate}
\item $\snorm{\Source^\star  -  \Source^{(k)}}$ is monotonically decreasing
\item $\lim_{k \to \infty}   \Source^{(k)} = \rec_{\theta}^{\rm null} (\Data)$
\item $\snorm{\Source^\star  -  \Source^{(k)}} \leq
\snorm{\Source^\star  - \rec_{\theta}^{\rm res}(\Data)}$.
\end{enumerate}
\end{enumerate}
\end{theorem}

\begin{proof}
Will be presented elsewhere.
\end{proof}

Theorem~\ref{thm:null} implies that iteration \eqref{eq:iter1}, \eqref{eq:iter2}
defines   a sequence
\begin{equation} \label{eq:iter3}
	\rec_{\theta}^{\mathrm{null},(k)}(\Data) \triangleq \Source^{(k)}
\end{equation}
that monotonically converges to $\rec_{\theta}^{\rm null}(\Data) $.
It implies that the nullspace network as well as the approximate nullspace network
$\rec_{\theta}^{\mathrm{null},(k)}(\Data)$  have  a smaller reconstruction
error than the residual network.  Moreover, according to Theorem~\ref{thm:null}
the  nullspace network  yields a solution of the equation $\Fullop \Source =\Data$
even for elements very  different from the  training data.

\begin{figure}[htb!]
\includegraphics[width=0.45\columnwidth]{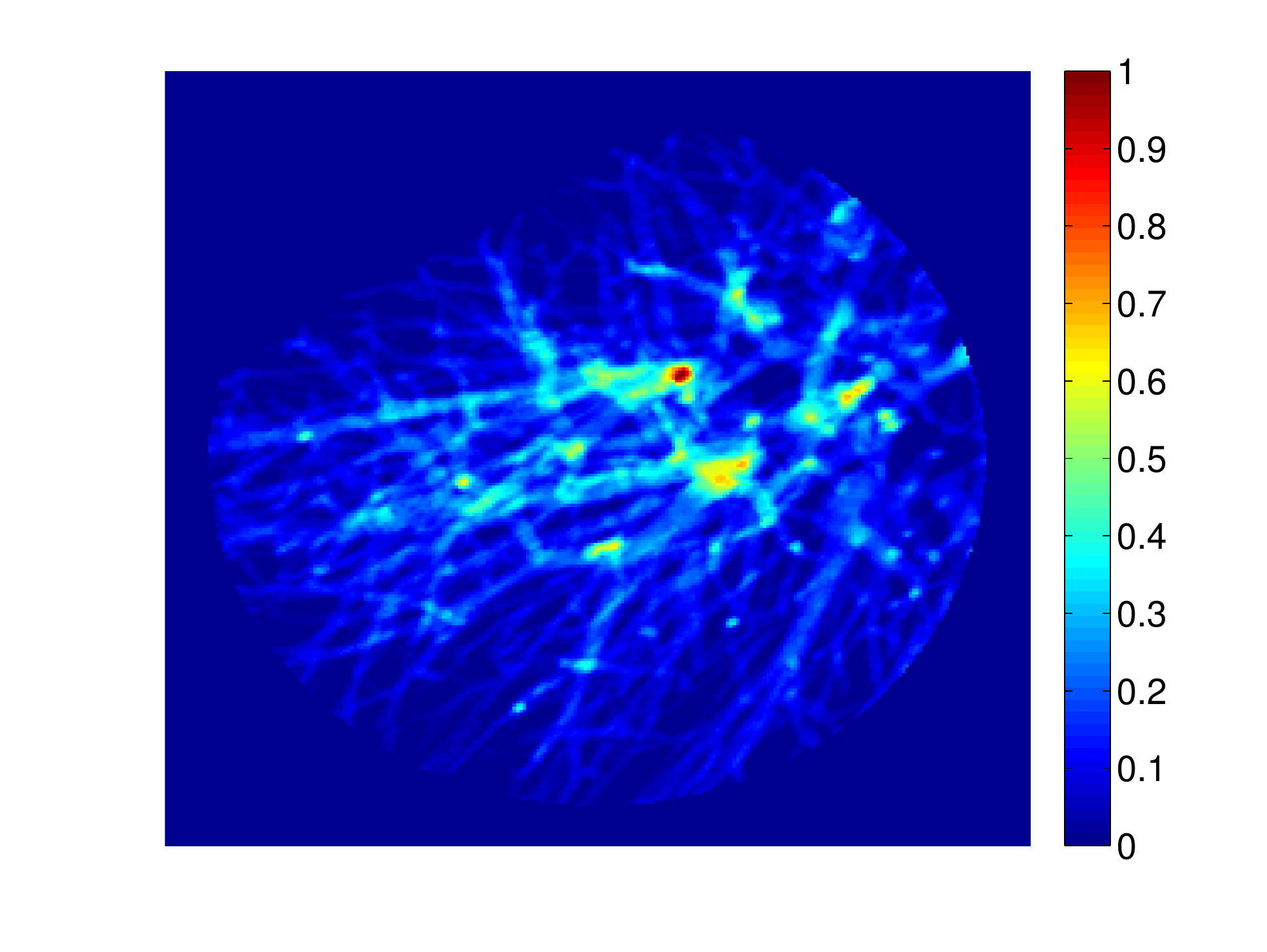}
\includegraphics[width=0.45\columnwidth]{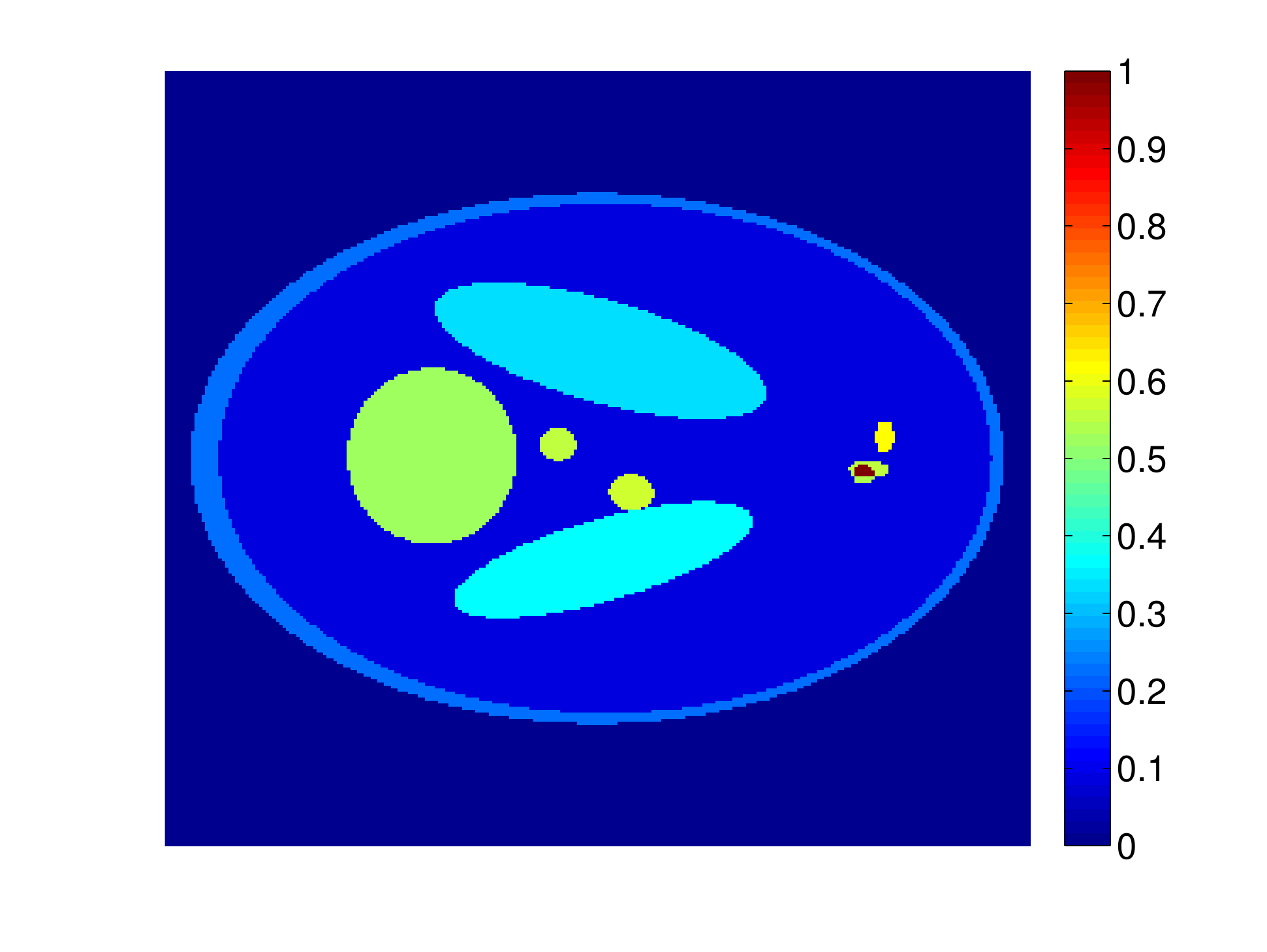}\\
\includegraphics[width=0.45\columnwidth]{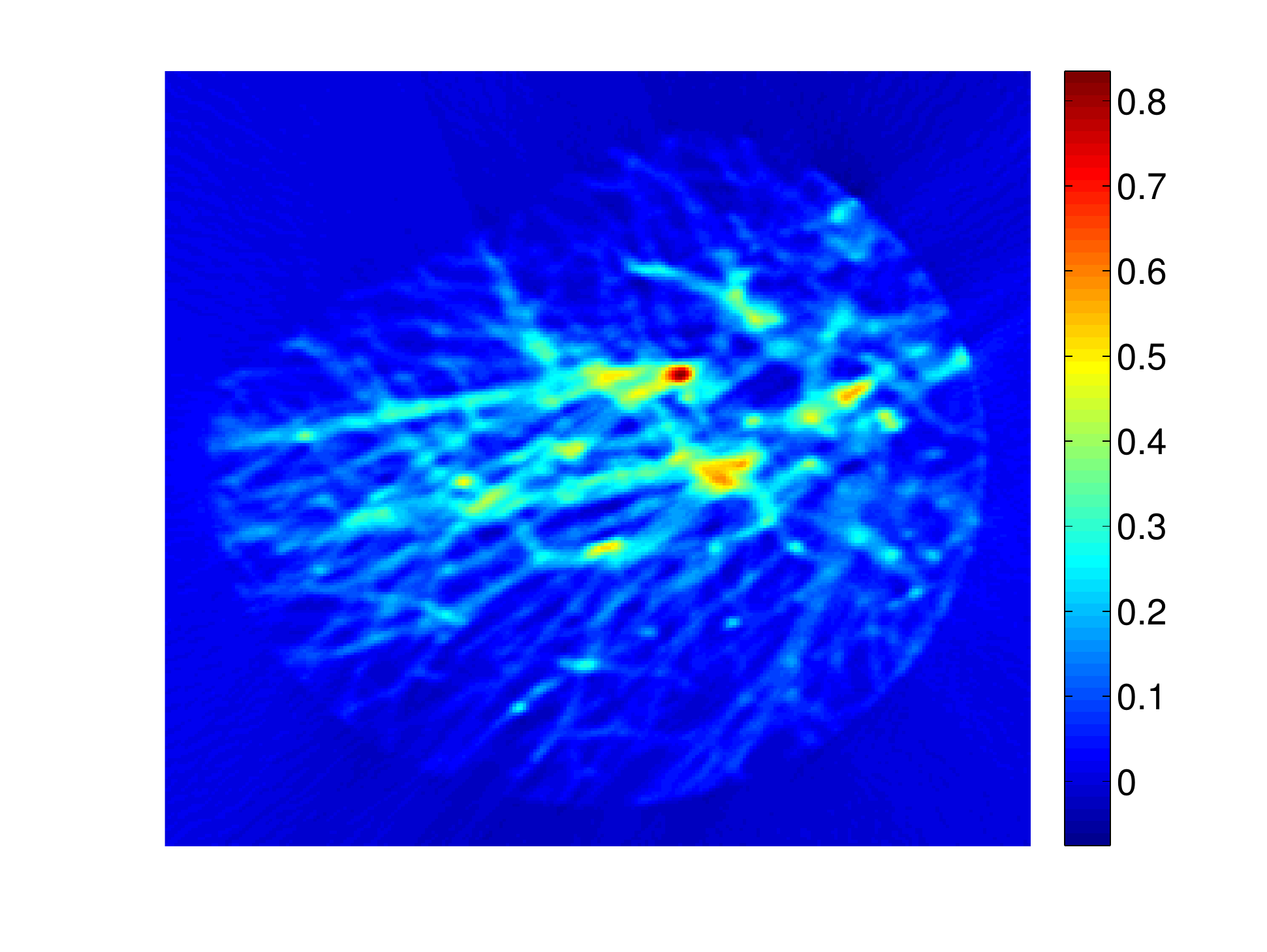}
\includegraphics[width=0.45\columnwidth]{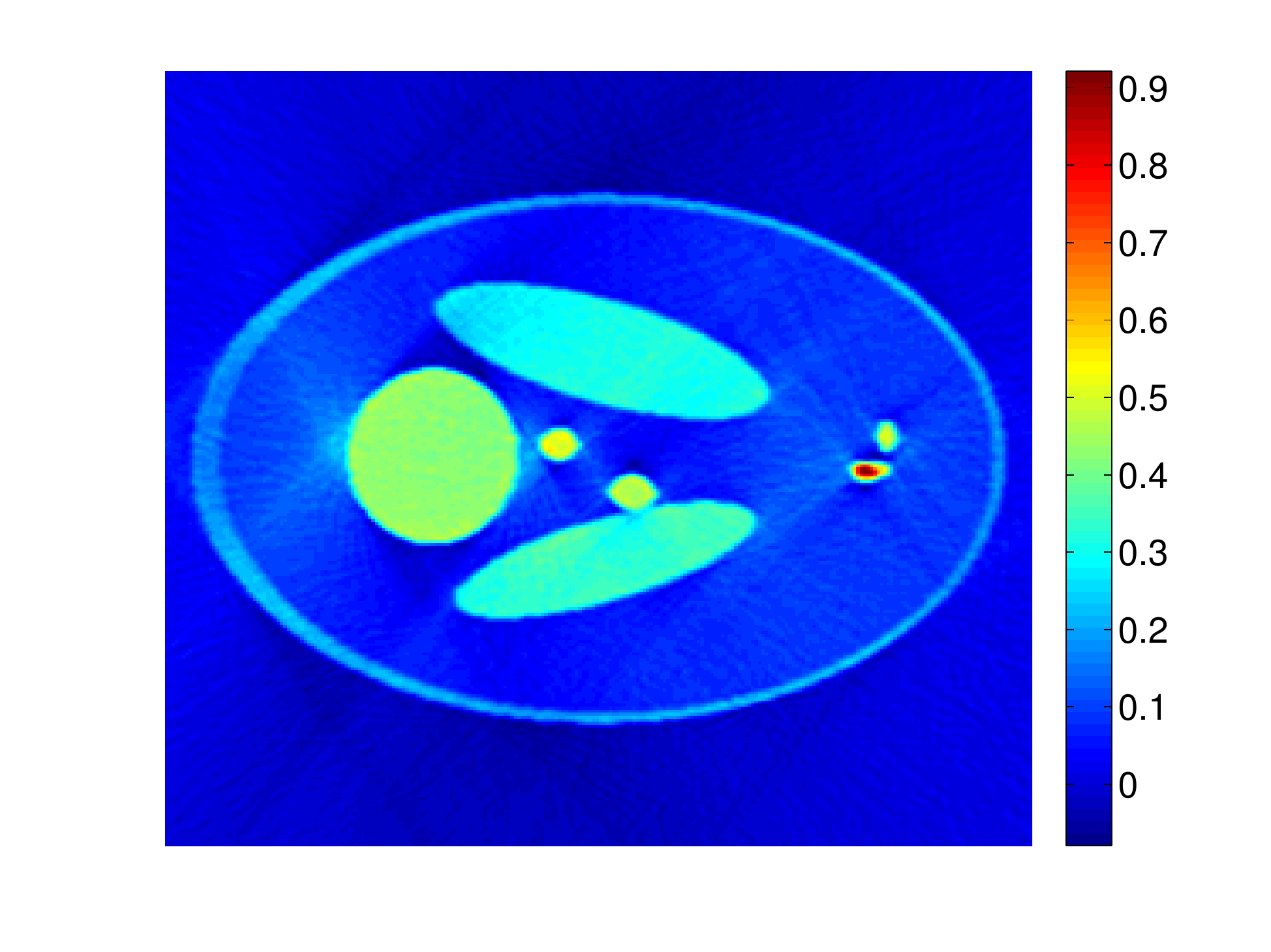}
\caption{Test phantoms for results presented below.
Top: vessel phantom (left) and  head phantom (right).
Bottom: FBP reconstruction from full data of vessel phantom (left)
and head phantom (right).\label{fig:phantoms}}
        \end{figure}

\section{Numerical results}
\label{sec:num}

In this section we numerically compare the joint $\ell^1$-minimization  approach
with the residual network and the nullspace network.  We also compare the results with plain
FBP.  We use Keras~\cite{keras} with TensorFlow~\cite{tensorflow}
to train and evaluate the CNN. The FBP, the  $\ell^1$-minimization  algorithm and the
iterative update \eqref{eq:iter2} is  implemented  in MATLAB.
We ran all our experiments on a computer using an
Intel i7-6850K and an NVIDIA 1080Ti.   The phantoms as well as the FBP
reconstruction from fully sampled data  are shown in
Figure~\ref{fig:phantoms}. Note that we use limited view data which implies that
the reconstructions in Figure~\ref{fig:phantoms} contain some artefacts.

\subsection{Measurement setup}
\label{eq:setup}

The entries of any  discrete  PA source  $\Source \in \R^{n}$ with
$n= 256^2$ correspond to  discrete samples  of the continuous source at a
$256 \times   256$ Cartesian grid covering the square $[\SI{-5}{\micro \meter} , \SI{9}{\micro\meter}] \times
[ \SI{-12.5}{\micro \meter}, \SI{1.5}{\micro \meter}]$.
The full wave data $\Data  \in \R^{ M P}$ corresponds to $P= 747$ equidistant
temporal samples in $[0,T]$ with $T=\SI{4.9749d-2}{\micro \second} $ and $M=240 $ equidistant sensor locations on the circle of radius $\SI{40}{\micro \meter}$ and polar angles in the interval  $[\SI{35}{\degree}, \SI{324}{\degree}] $. The sound speed is taken as   $c= \SI{1.4907d3}{\meter \per \second}$.
 The wave equation is evaluated by  discretizing the solution formula of the  wave equation, and the
inversion formula \eqref{eq:fbp2d} is discretized using the     standard  FBP procedure  described in \cite{FinHalRak07,haltmeier2011mollification}. Recall  that the continuous setting the inversion integral in \eqref{eq:fbp2d} equals the adjoint of the forward operator. Therefore the above procedure  gives a  pair
\begin{align*}
\Wave \colon \R^{n} \to \R^{mQ} \\
\BP \colon \R^{nQ} \to \R^{n}
\end{align*}
of  forward operator and unmatched adjoint.

We consider $m = 60$ spatial measurements which corresponds to a
compression  factor of four. For the sampling matrices $\samp \in \R^{m \times M}$ we use
the following instances:
\begin{itemize}
\item Deterministic sparse  subsampling matrix
with
entries
\begin{equation} \label{eq:subsampling}
\samp[i,j]
=
\begin{cases}
2 & \text{ if } j = 4(i-1) + 1 \\
0 & \text{ otherwise} \,.
\end{cases}
\end{equation}
\item Random Bernoulli matrix  where each entry is taken
independently as  $\pm 1/\sqrt{m}$  with equal  probability.
\end{itemize}

The Bernoulli matrix  satisfies then RIP  with high probability,
whereas the sparse subsampling matrix doesn't. Therefore, we  expect the
$\ell^1$-minimization  approach to work better for  Bernoulli measurements.
On the other hand,  in the subsampling case the  artefacts  have more structure
which therefore is expected  to be better   for the  deep  learning approaches.
Our findings below confirm these conjectures.

\subsection{Construction of reconstruction networks}
\label{sec:training}

For the residual and the nullspace network
we use the backprojection layer
\begin{equation}\label{eq:bo}
\Fullop^\sharp =  \BP  \circ   \Samp^\trans \,,
\end{equation}
and the  same trained CNN.
For that purpose, we construct $N = 5000$ training examples
$ (\Xin_k,\Xout_k)_{k=1}^N$ where $\Xout_k$ are taken as projection
images from three dimensional lung blood vessel data  as described in \cite{schwab2018fast}.
All images  $\Xout_k$  are  normalized to have maximal intensity one.
The corresponding input images
are computed by    $\Xin_k  =   \Fullop^\sharp \Fullop    \Xout_k $.
The CNN is constructed by minimizing the  mean absolute error
\begin{equation} \label{eq:err-nn}
	E_N(\theta)
	\triangleq
\frac{1} {N}
\sum_{k=1}^N \norm{ ( \Id + \UU_\theta) (\Xin_k) - \Xout_k }_1
\end{equation}
using  stochastic  gradient descent   with  batch size 1 and momentum 0.9. We trained for 200 epochs
and used a  decaying learning parameter between  $0.005$ to $0.0025$.

Having computed the minimizer of \eqref{eq:err-nn} we use the trained residual network
$\rec_{\theta}^{\rm res}$ as well as the corresponding approximate nullspace network
$\rec_{\theta}^{\mathrm{null}, (10)}$ for image reconstruction.

\subsection{Blood vessel phantoms}

First we investigate the performance on  50 blood vessel phantoms
that are not contained in the training set. We consider sparse sampling
as well as Bernoulli measurements.  For the joint recovery approach,
 we use 70 iterations of the iterative thresholding procedure with coupling parameter
 $\al = 0.001$,  regularization parameter $\beta = 0.005$ and step size  $\mu = 0.125$.
 For the (approximate) nullspace network
$\rec_{\theta}^{\mathrm{null}, (10)}$  we use  10 iterations to approximately compute the projection.
Results for one of the vessel phantoms are visualized  in Figure \ref{fig:vessel}.
To quantitatively evaluate the results
 we computed the MSE (mean square error), the PSNR (peak signal to noise ratio)  and the SSIM (structural similarity index \cite{wang2004image}) averaged over all  50 blood vessel phantoms.
 The reconstruction errors are summarized in
 Table~\ref{tab:vessel} where  the best results are framed.

\begin{figure}[htb!]
\includegraphics[width=0.45\columnwidth]{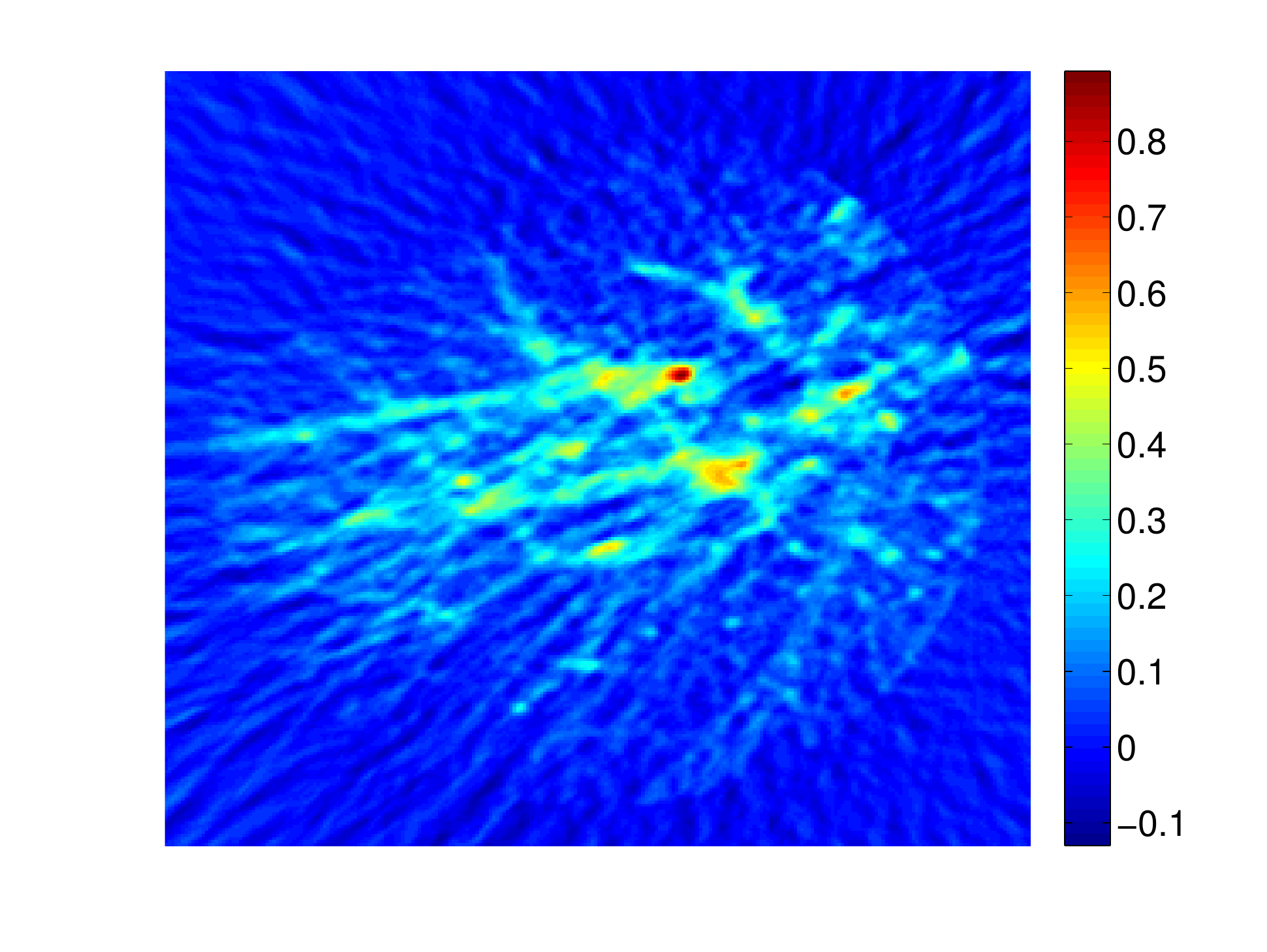}
\includegraphics[width=0.45\columnwidth]{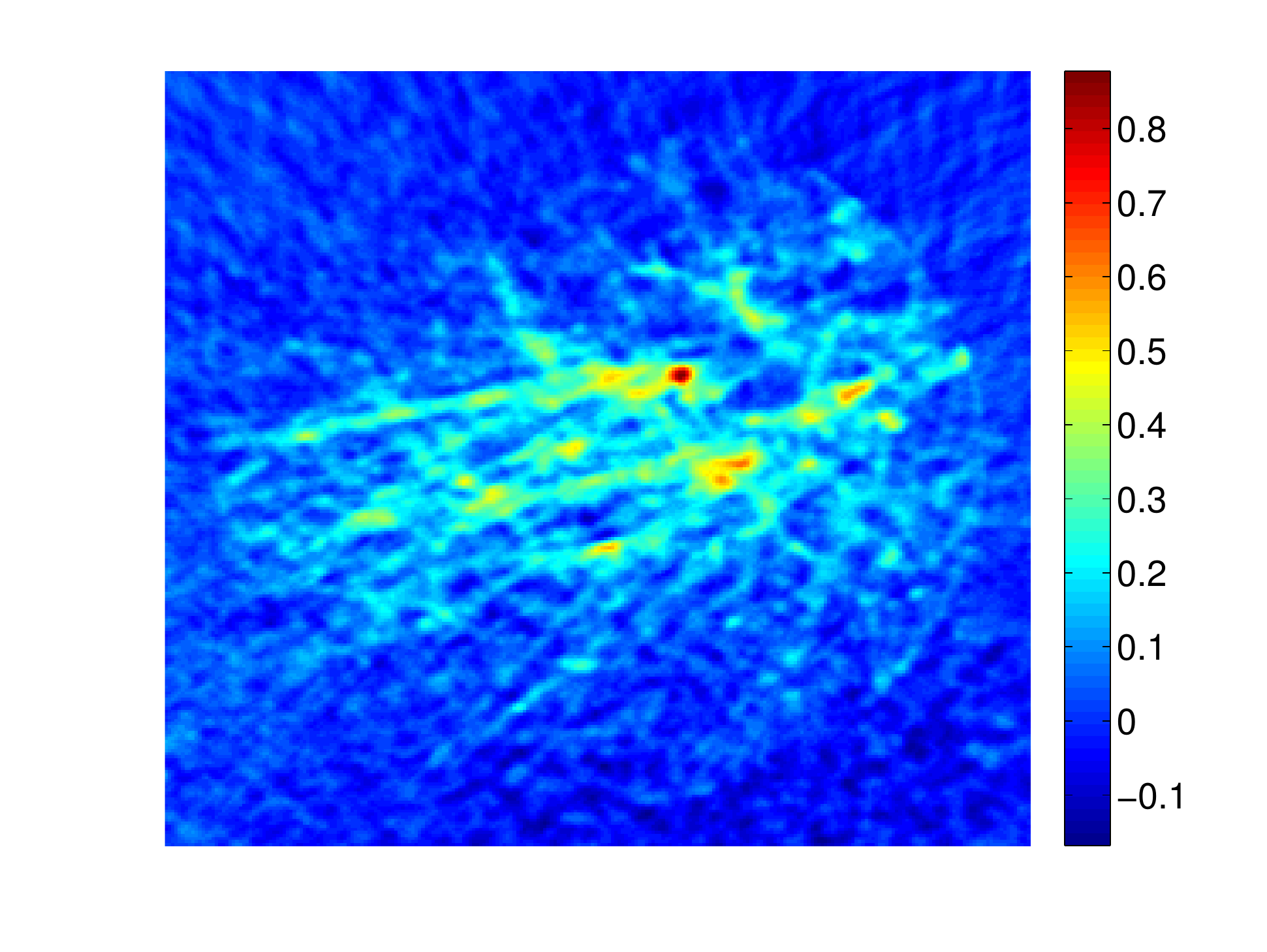}\\
\includegraphics[width=0.45\columnwidth]{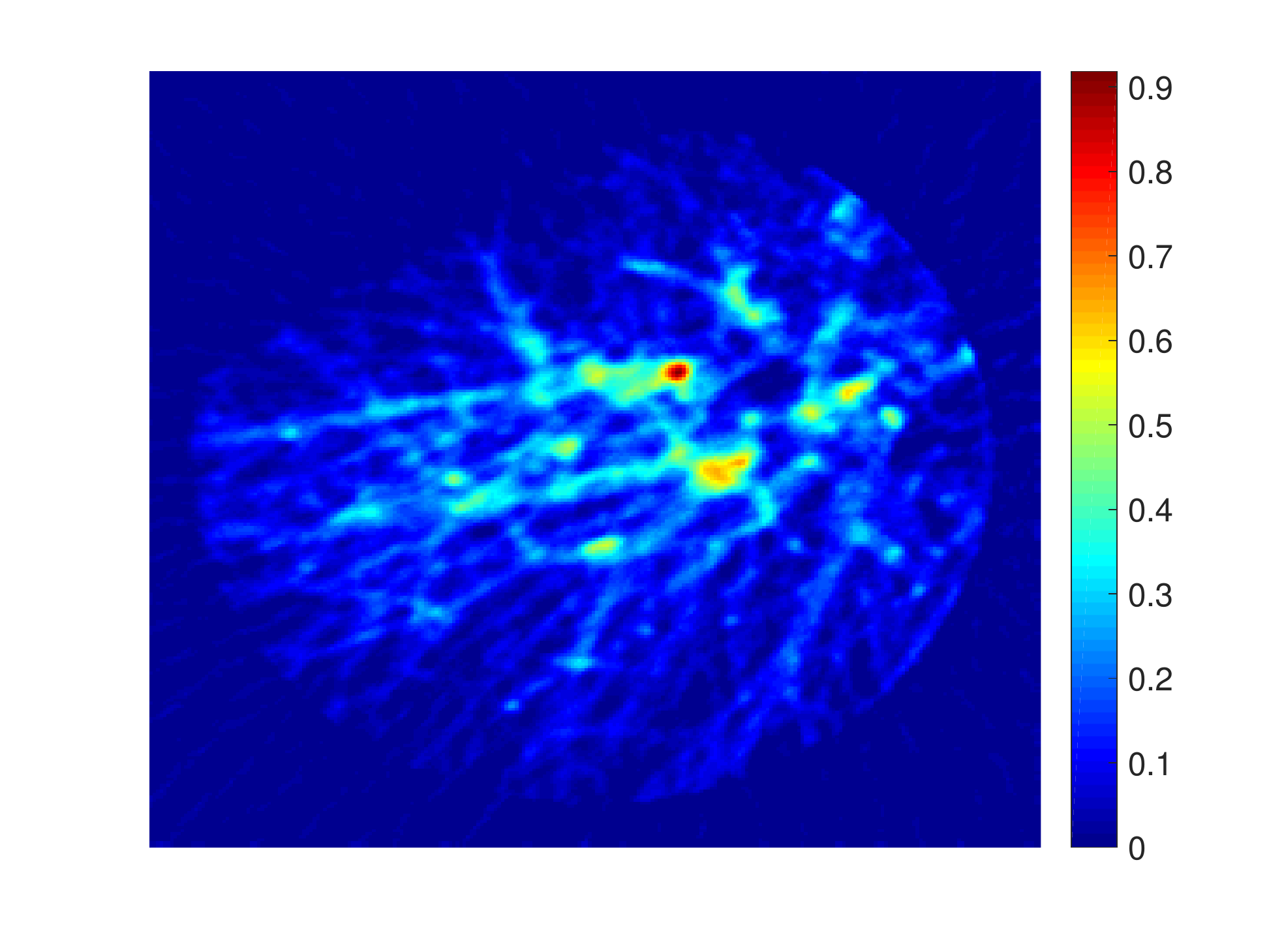}
\includegraphics[width=0.45\columnwidth]{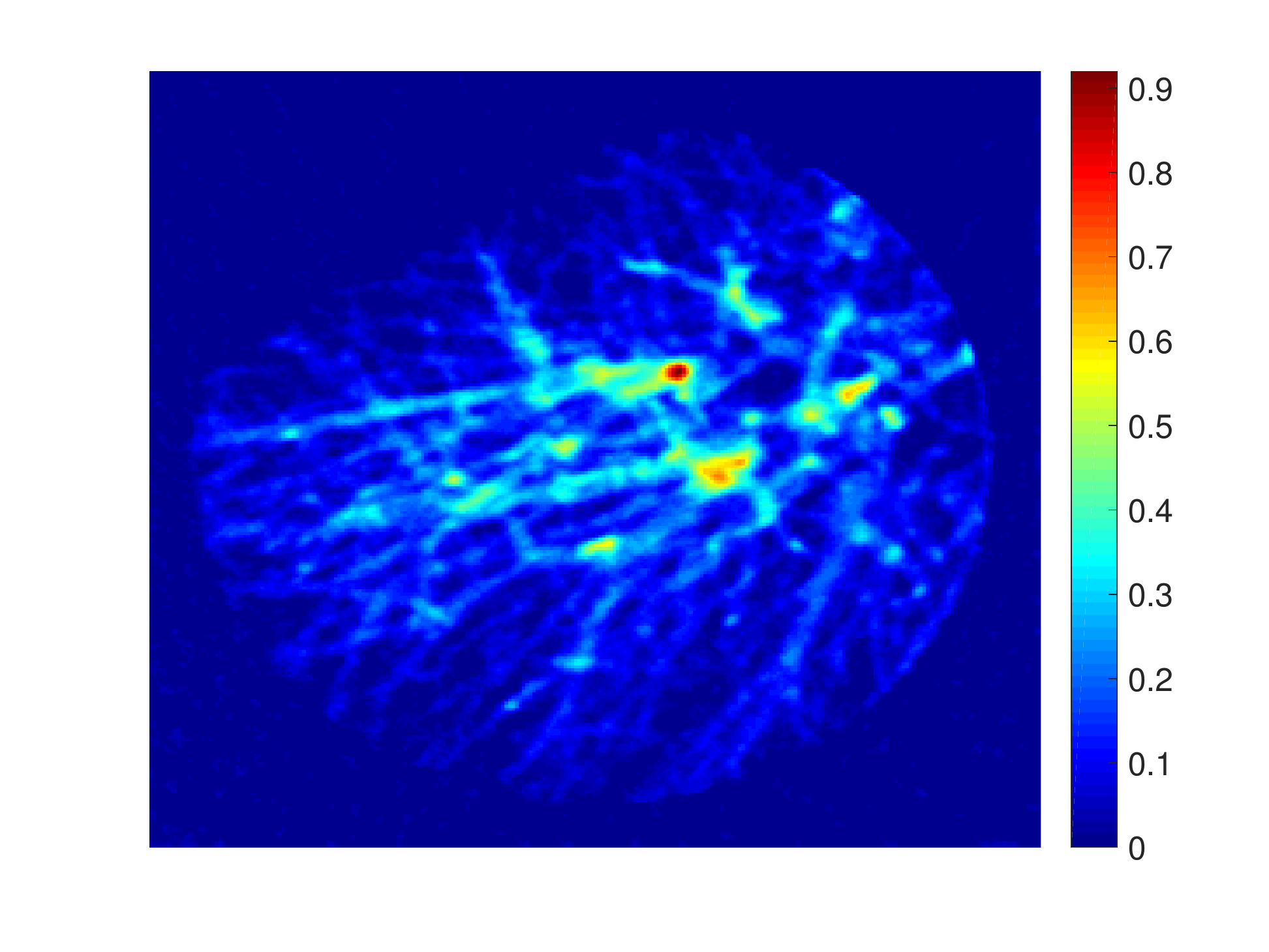}\\
\includegraphics[width=0.45\columnwidth]{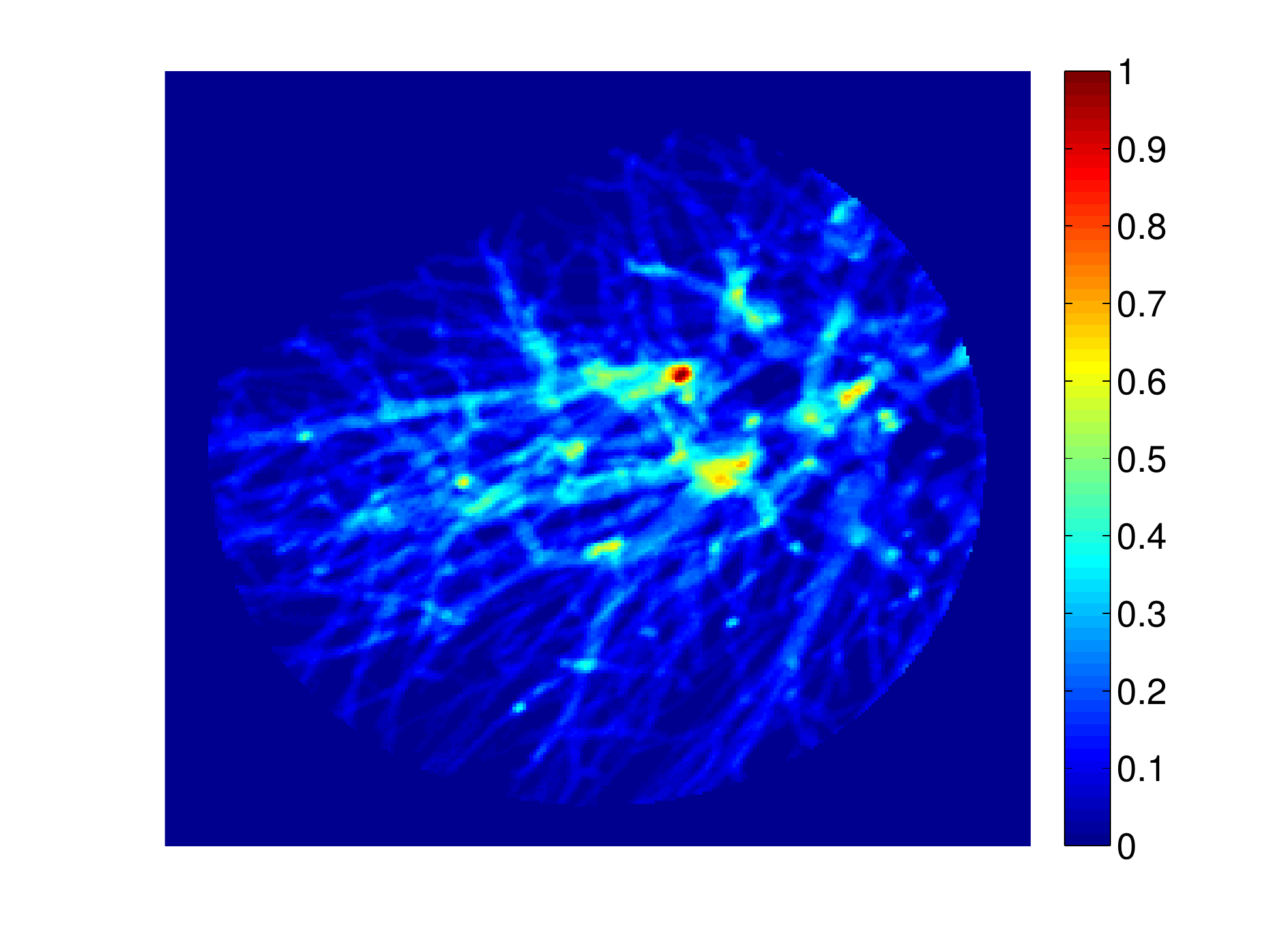}
\includegraphics[width=0.45\columnwidth]{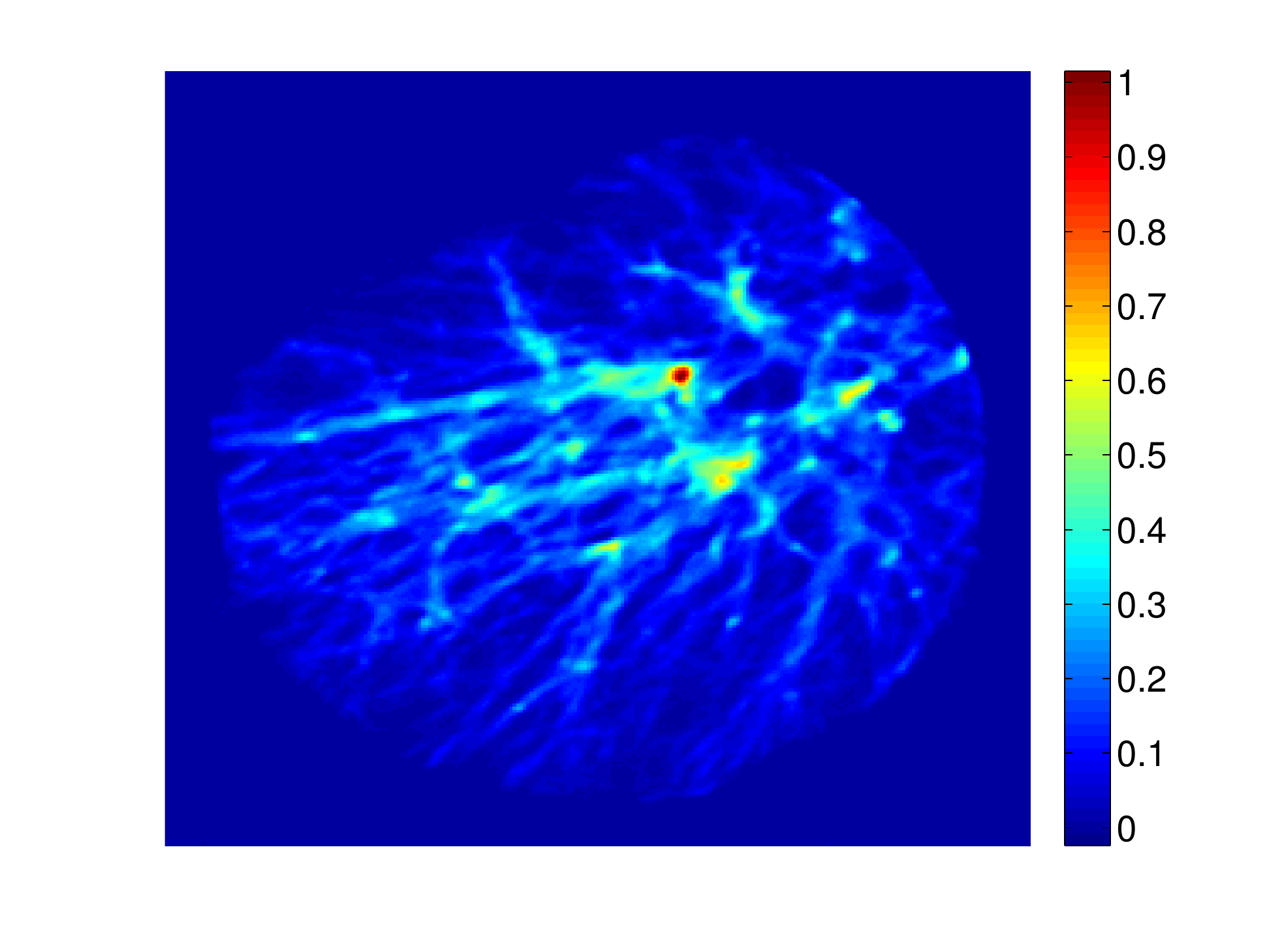}\\
\includegraphics[width=0.45\columnwidth]{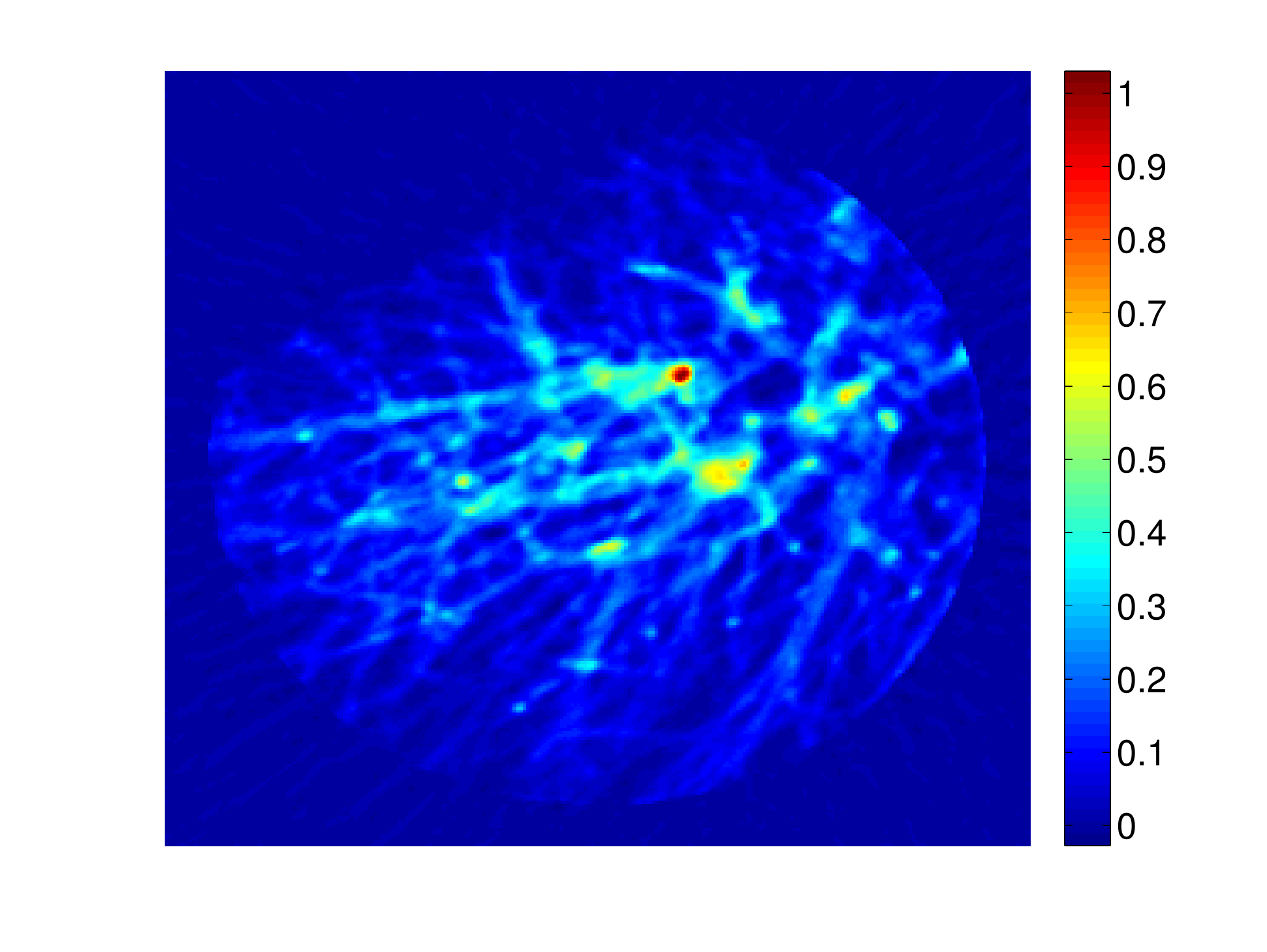}
\includegraphics[width=0.45\columnwidth]{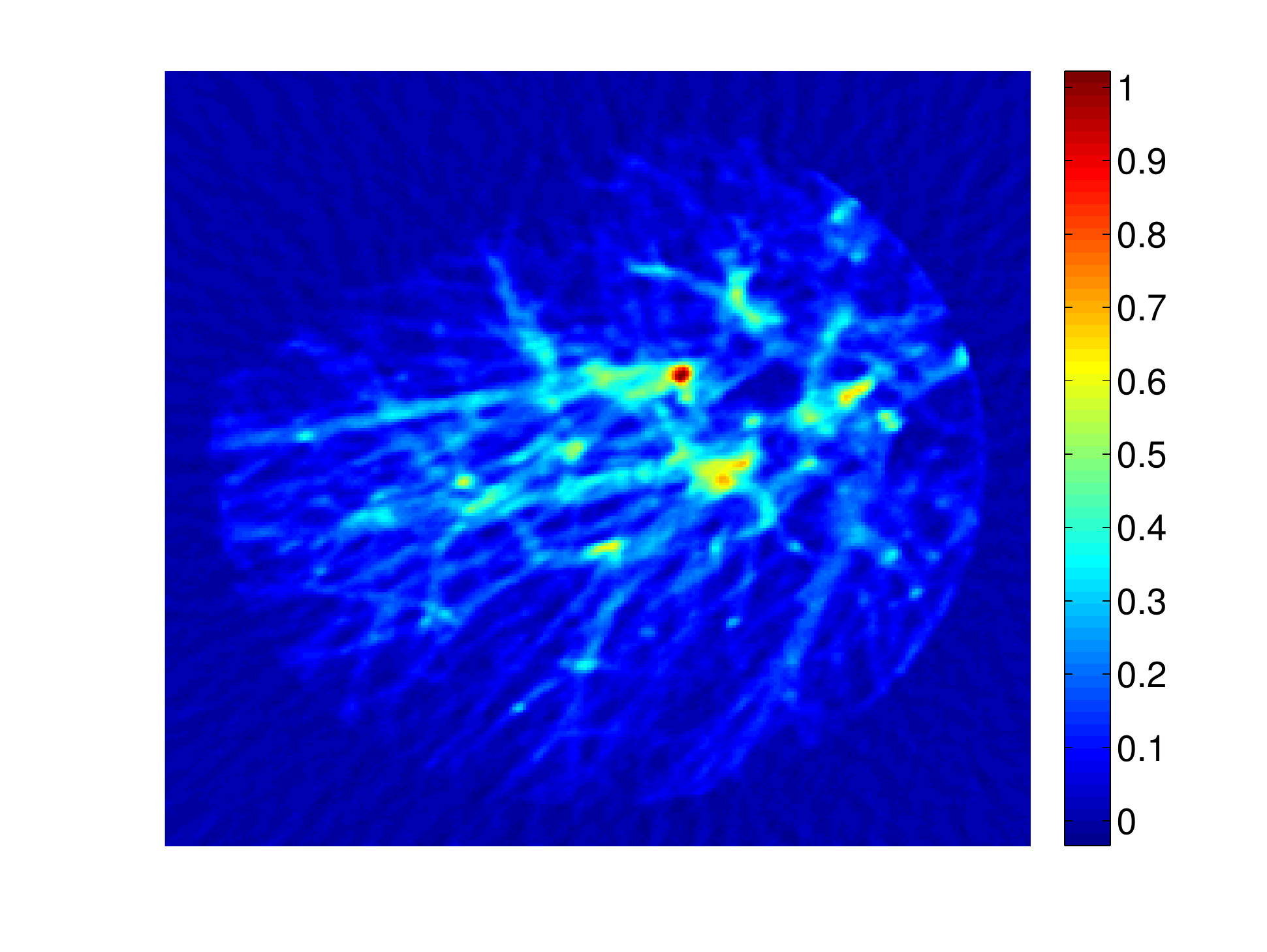}
        \caption{Reconstructions  of blood  vessel  image from sparse measurements (left) and Bernoulli measurements (right).
  Top row: FBP reconstruction.
  Second row: Joint $\ell^1$-minimization.
  Third row: Residual network.
  Bottom  row: Nullspace network.
  \label{fig:vessel}}
  \end{figure}

\begin{table}[htb!]
\caption{Performance averaged over 50 blood vessel images.\label{tab:vessel}}
\centering
\begin{tabular}{ l | c c c}
\toprule
      & SME & PSNR &  SSIM  \\
\midrule
 \multicolumn{2}{c}{Sparse measurements}\\
    FBP       & \colorbox{white}{\SI{15.1d-4}{}}  & \colorbox{white}{28.6} &  \colorbox{white}{\SI{4.83d-1}{}}   \\
    $\ell^1$-minimization             & \colorbox{white}{\SI{3.35d-4}{}}    &  \colorbox{white}{35.0}   &  \colorbox{white}{\SI{8.50d-1}{}}    \\
    residual network        & \colorbox{white}{\SI{3.08d-4}{}}   & \colorbox{white}{35.6} & \colorbox{green}{  \SI{9.30d-1}{} }     \\
    nullspace network &  \colorbox{green}{ \SI{2.22d-4}{}}  & \colorbox{green}{ 37.0 }& \SI{9.17d-1}{}     \\
\midrule
\multicolumn{2}{c}{Bernoulli measurements}\\
    FBP       &\colorbox{white}{ \SI{20.2d-4}{}}  & \colorbox{white}{27.3}  &  \colorbox{white}{\SI{4.18d-1}{}}    \\
    $\ell^1$-minimization          &  \colorbox{green}{  \SI{1.89d-4}{} }     & \colorbox{green}{  37.5 }  & \colorbox{green}{  \SI{9.06d-1}{} }    \\
    residual network        & \colorbox{white}{\SI{6.32d-4}{}}  & \colorbox{white}{32.6} & \colorbox{white}{\SI{8.89d-1}{}}     \\
    nullspace network & \colorbox{white}{\SI{2.21d-4}{}}  & \colorbox{white}{36.9}  & \colorbox{white}{\SI{8.89d-1}{}}     \\
  \bottomrule
\end{tabular}
\end{table}

From Table~\ref{tab:vessel} we see that  the hybrid
as well as the deep learning based methods significantly
outperform the FBP reconstruction.
Moreover, the deep learning approach even outperforms
the joint recovery approach for the sparse sampling.
The nullspace network in all cases decreases the MSE (increases the PSNR)
compared to the residual network.

\subsection{Shepp-Logan type phantom}

Next we  investigate the performance on a Shepp-Logan type phantom
that contains structures    completely  different from the training data.
For the joint recovery approach, we use   50  iterations of the iterative
thresholding procedure with
 $\al = 0.001$  regularization parameter $\beta = 0.005$ and step size  $\mu = 0.1$.
 For the nullspace network we use
$\rec_{\theta}^{\mathrm{null}, (10)}$.  Results  are shown in
 Figure~\ref{fig:head}.    Table~\ref{tab:head} shows the MSE,  the PSNR  and the SSIM
 for the head phantom,  where  the best results are again framed.

\begin{figure}[htb!]
\includegraphics[width=0.45\columnwidth]{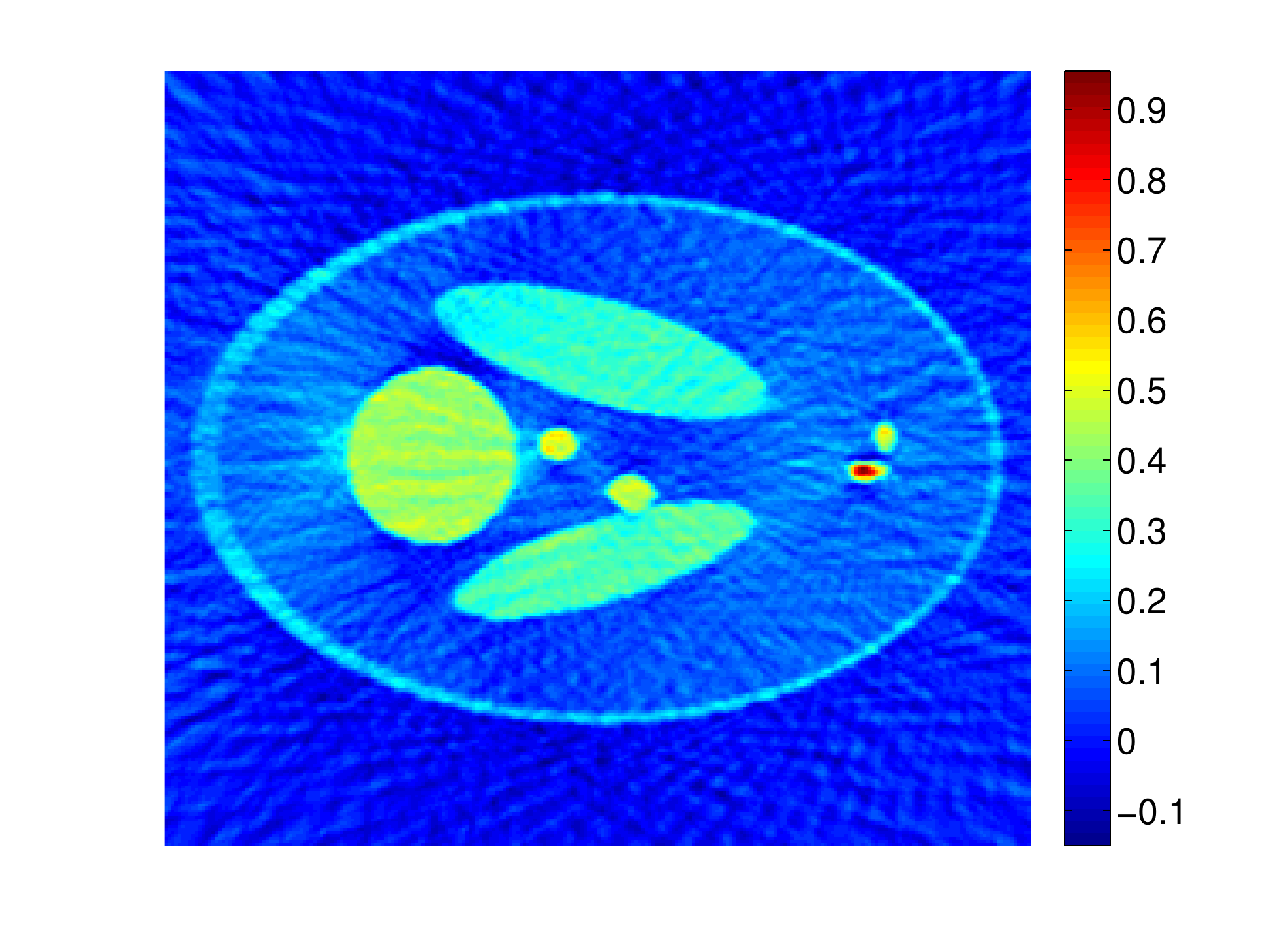}
\includegraphics[width=0.45\columnwidth]{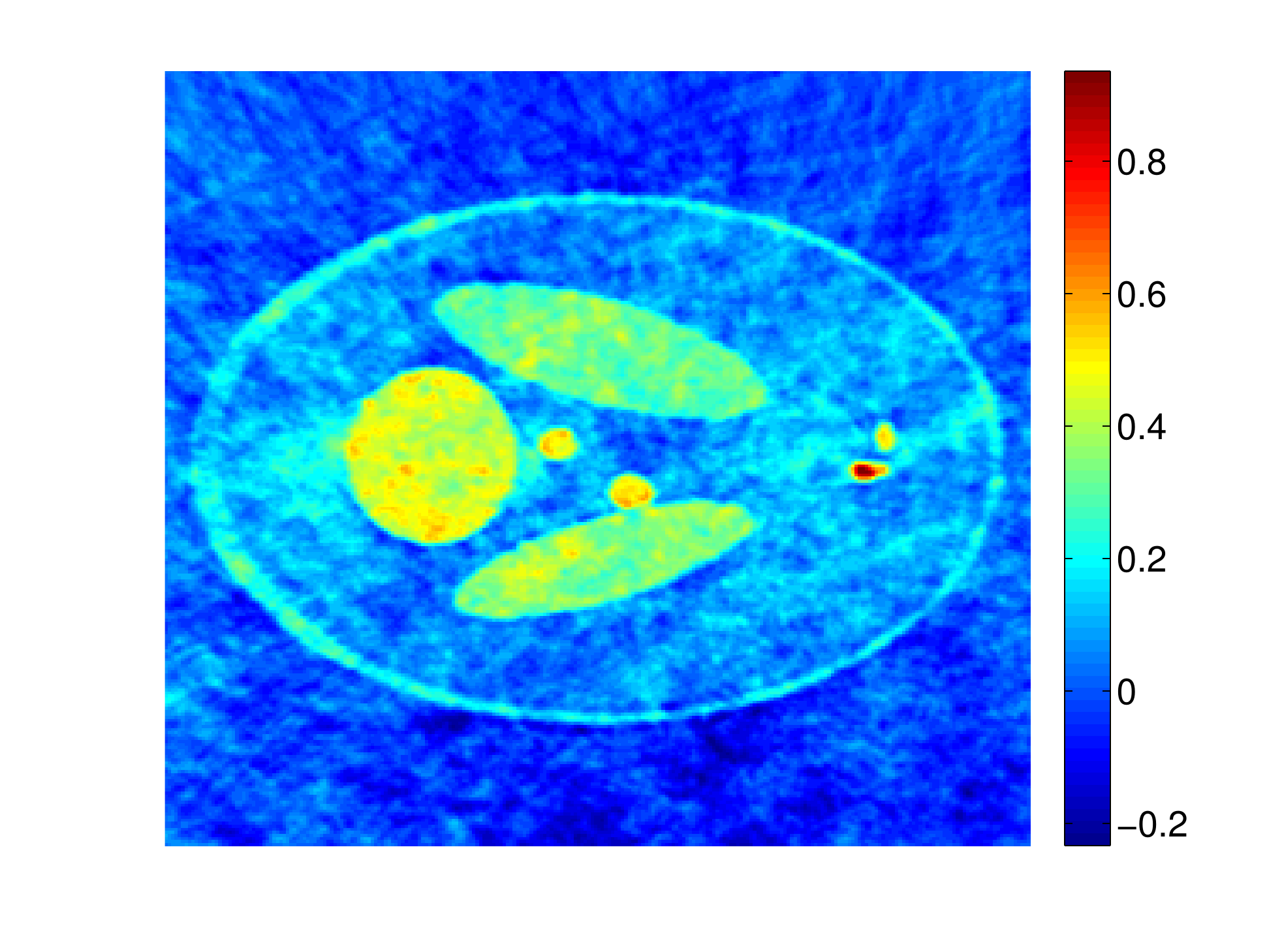}\\
\includegraphics[width=0.45\columnwidth]{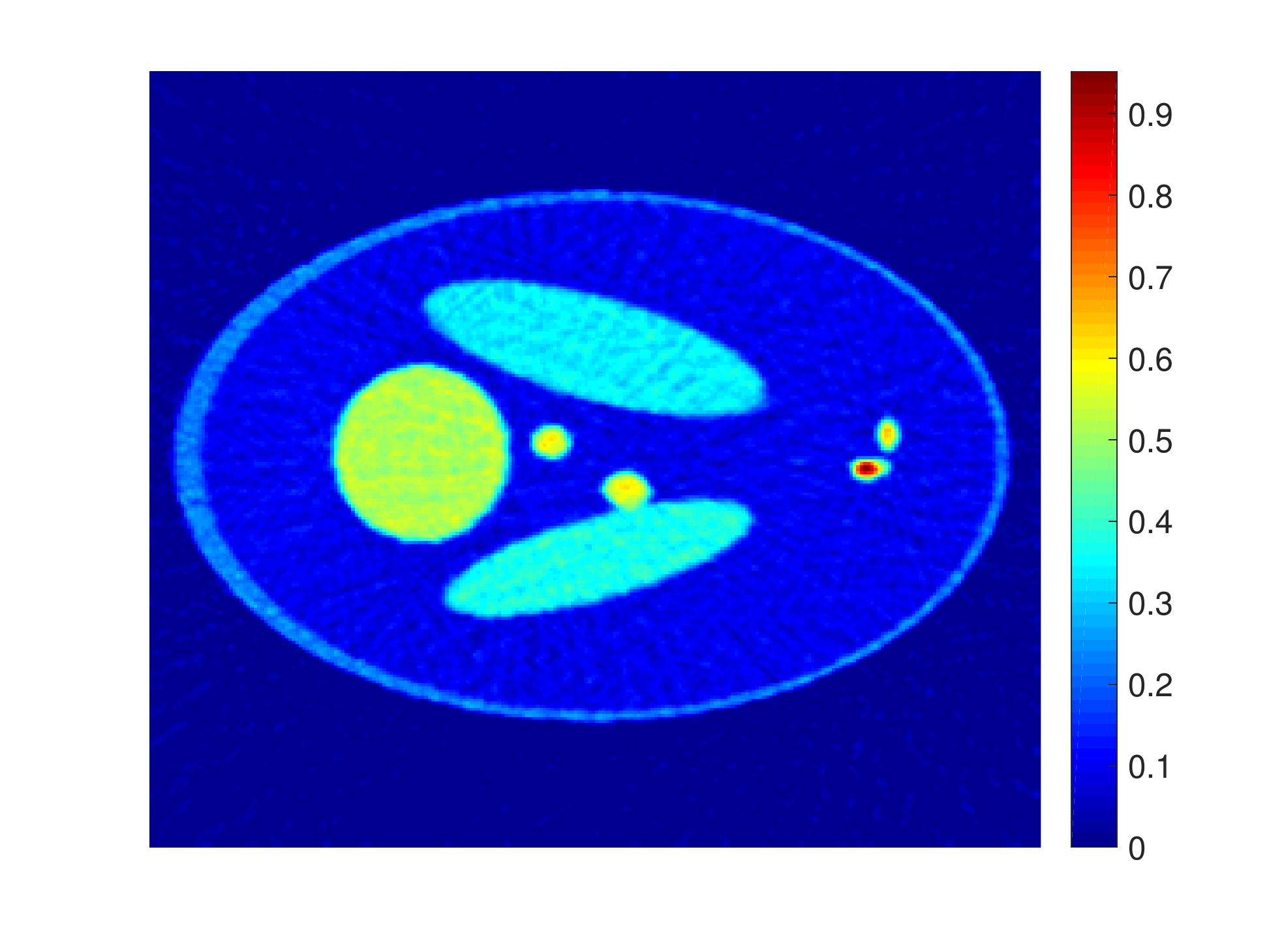}
\includegraphics[width=0.45\columnwidth]{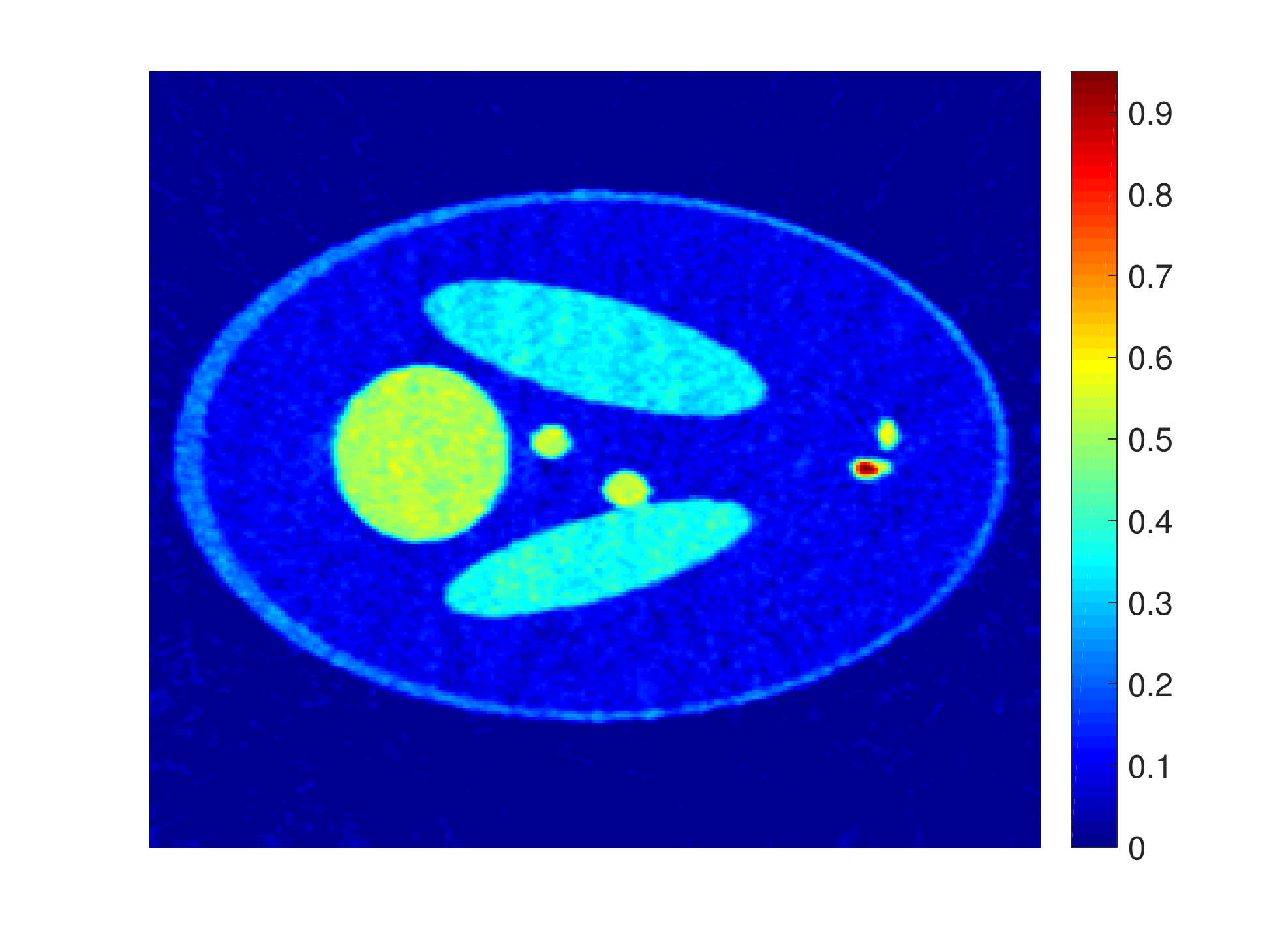}\\
\includegraphics[width=0.45\columnwidth]{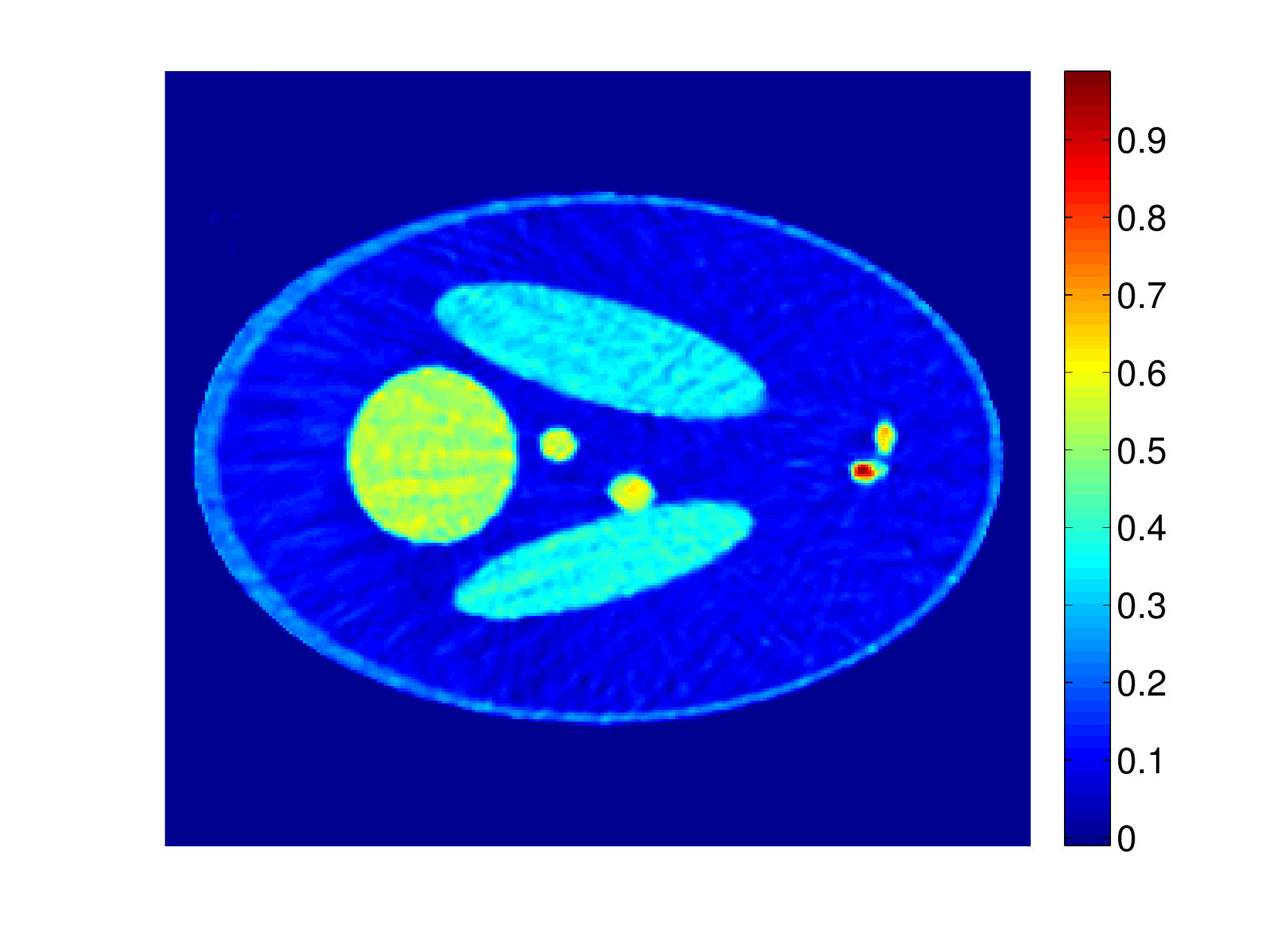}
\includegraphics[width=0.45\columnwidth]{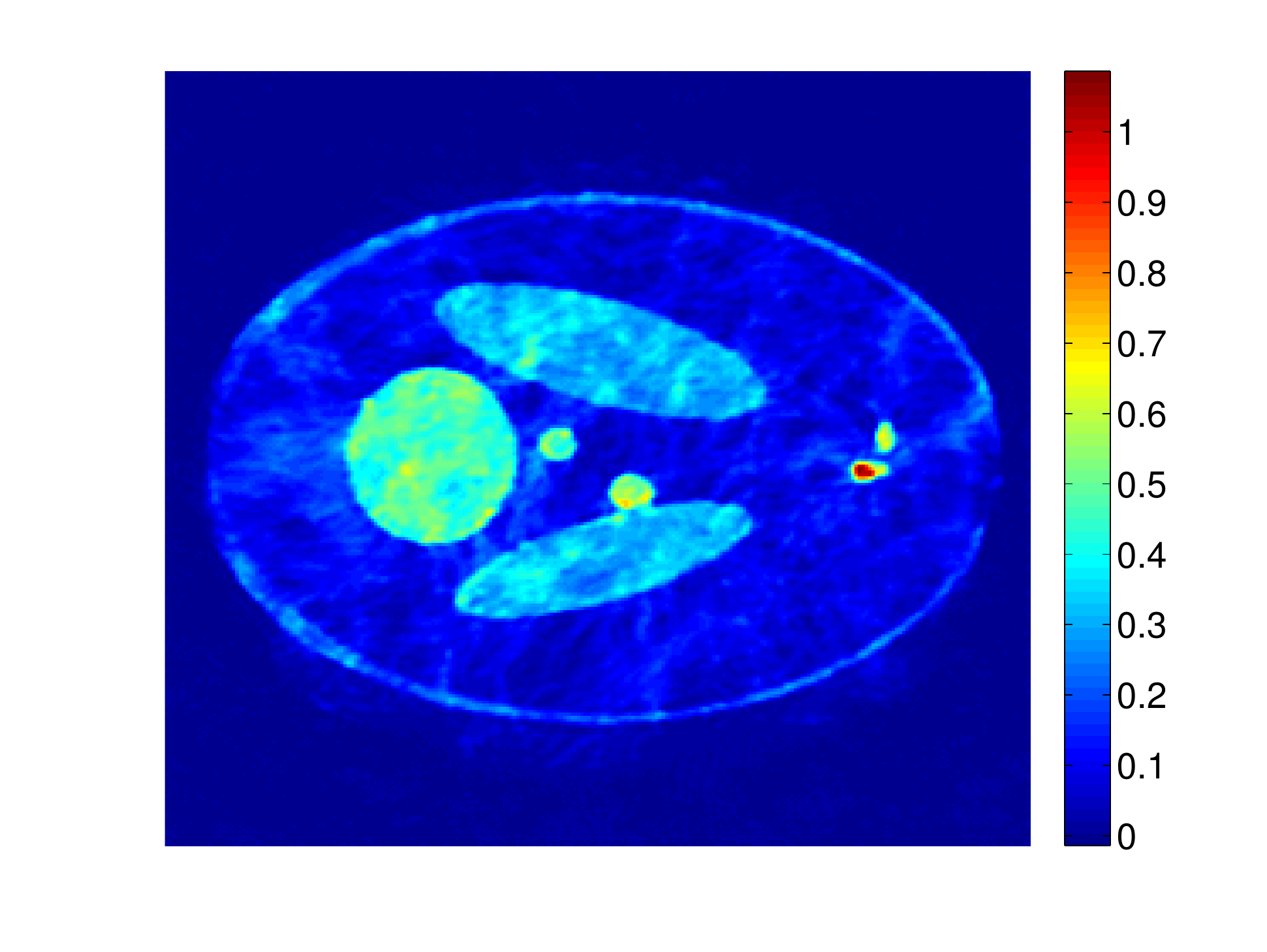}\\
\includegraphics[width=0.45\columnwidth]{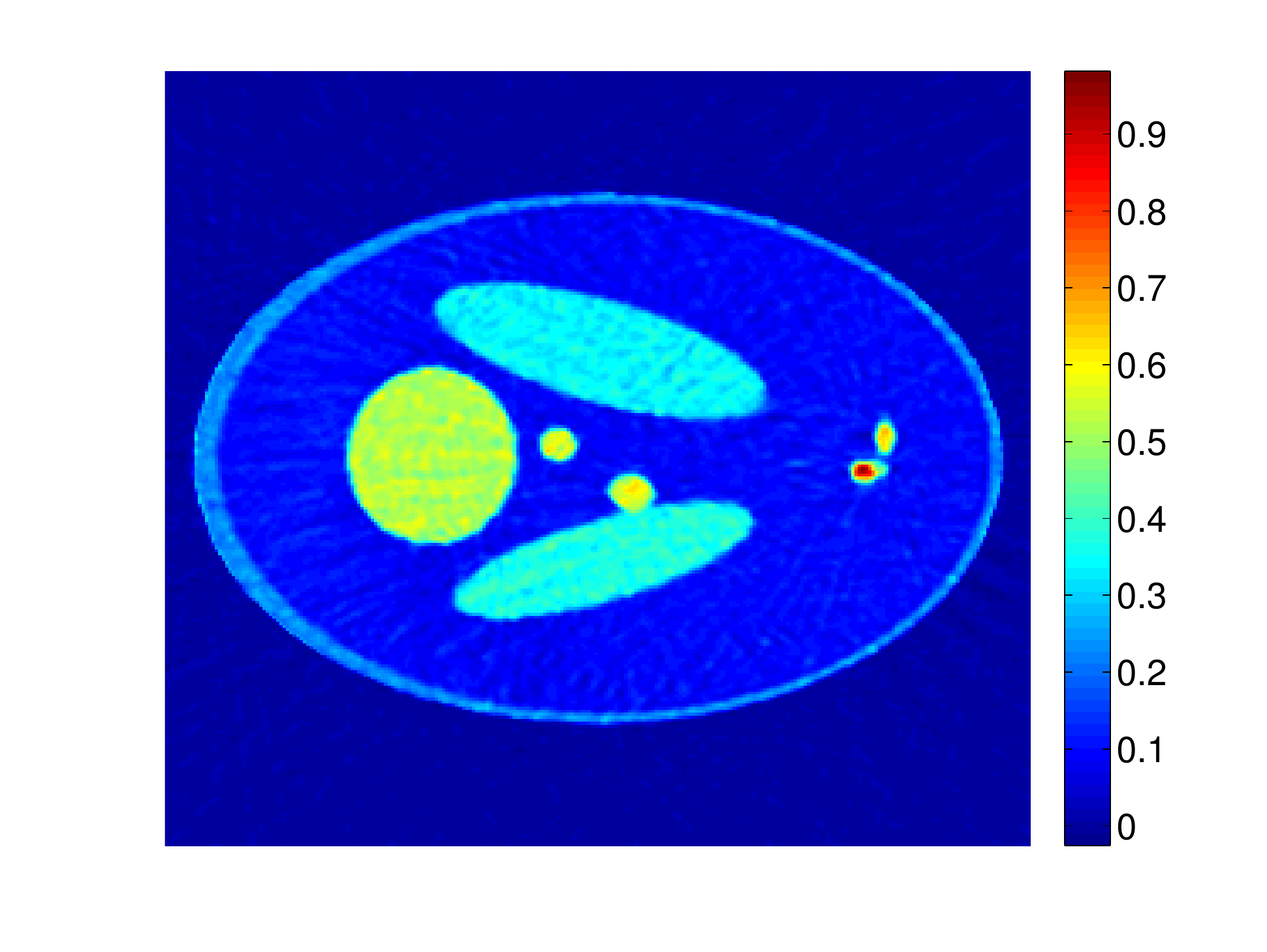}
\includegraphics[width=0.45\columnwidth]{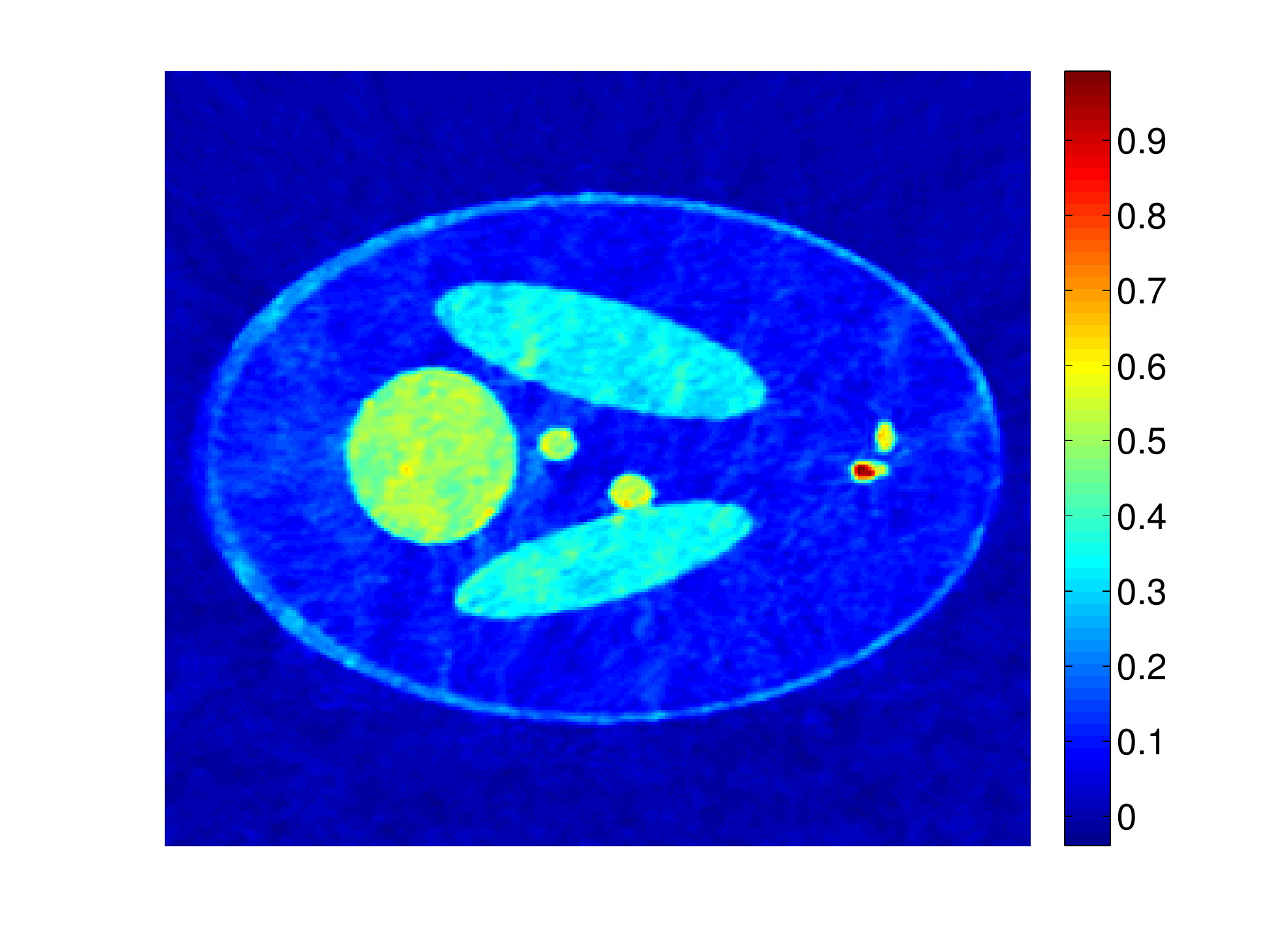}
        \caption{Reconstructions  of  Shepp-Logan type phantom  from sparse measurements (left) and Bernoulli measurements (right).
  Top row: FBP reconstruction.
  Second row: Joint $\ell^1$-minimization.
  Third row: Residual network.
  Bottom  row: Nullspace network.
  \label{fig:head}}
        \end{figure}

        \begin{table}[htb!]
\caption{Performance for the Shepp-Logan phantom.\label{tab:head}}
\centering
\begin{tabular}{ l | c  c  c}
\toprule
      & SME & PSNR &  SSIM  \\
\midrule
\multicolumn{2}{c}{Sparse measurements}\\
    $\ell^1$-minimization             &  \colorbox{white}{\SI{6.73d-4}{}}     &  \colorbox{white}{31.7}   &  \colorbox{white}{\SI{8.04d-1}{}}     \\
    residual network        &    \colorbox{white}{\SI{6.32d-4}{}}     & \colorbox{white}{32.0}   & \colorbox{green}{  \SI{8.90d-1}{} }      \\
    nullspace network &   \colorbox{green}{  \SI{5.29d-4}{} }    & \colorbox{green}{  32.8 }  & \colorbox{white}{\SI{8.59d-1}{}}       \\
\midrule
\multicolumn{2}{c}{Bernoulli measurements}\\
    $\ell^1$-minimization             &  \colorbox{green}{   \SI{6.03d-4}{} }    & \colorbox{green}{  32.2 }  & \colorbox{green}{  \SI{8.19d-1}{} }     \\
    residual network        &    \colorbox{white}{\SI{19.2d-4}{} }     & \colorbox{white}{ 27.2 }  & \colorbox{white}{ \SI{7.63d-1}{} }       \\
    nullspace network &   \colorbox{white}{ \SI{6.92d-4}{}  }   & \colorbox{white}{ 31.6 }  & \colorbox{white}{ \SI{7.67d-1}{}  }     \\
  \bottomrule
\end{tabular}
\end{table}

 As the considered Shepp-Logan type phantom is  very different from
 the training data it is not surprisingly the standard residual network
 does not perform that well for the Bernoulli measurements.
 Surprisingly the residual network still works well in the sparse data case.
 The nullspace network yields significantly improved
 results compared for the residual network, especially for the
 Bernoulli case. In the Bernoulli case,  the $\ell^1$-minimization approach
 performs best, however only slightly better than the approximate nullspace
 network.

\section{Conclusion}
\label{sec:conclusion}

 In this paper we compared    $\ell^1$-minimization
 with deep learning  for CS PAT image reconstruction.
 The two approaches have been tested  on  blood vessel
 data (test data not contained in the training set that consists of similar  objects)
 as well as a Shepp-Logan type phantom (with structures very  different  from the training data).
 For the CS PAT measurements we   considered  deterministic subsampling
 as well as random Bernoulli measurements.  For the  used reconstruction networks, we considered
 the Unet with residual connection and an approximate nullspace network which contains an
 additional  data consistency layer.

In terms of  reconstruction quality, our findings can be summarized as follows:
\begin{enumerate}
\item
Sparse recovery and deep learning both significantly outperforms filtered backprojection
for both measurement matrices.
If the training data are not accurate for the object to be reconstructed, for  the deep learning
approach this conclusion only holds for the null-space network.

\item
In the case of the sparse measurement matrix, the deep learning  approach outperforms
$\ell^1$-minimization. In the case of Bernoulli measurement, the $\ell^1$-minimization  algorithms
yields better performance.

\item
The nullspace network contains a data consistence  layer  and yields good results even
for phantoms very different from the training data. Even for the test data similar to the training
 data it yields an improved PSNR compared to the residual network
\end{enumerate}

According to the above results we can recommend the $\ell^1$-minimization  algorithm in the   case of random measurements  and  the nullspace  network in the case of sparse  measurements. 
We point out that application of the CNN  only takes fractions of second
(actually, less than $0.01$ seconds) in Keras whereas the joint recovery approach requires around 
2 minutes for 50  iterations in Matlab.  Note that  this comparison is not completely 
fair and with a recent   GPU implementation  in PyTorch  we have been able the reduce the 
computation time  to about one second for 50 iterations. 
Nevertheless, the deep learning based methods are still significantly faster. 
Therefore,  especially the nullspace network is very promising for high quality real-time CS PAT 
imaging.

\section*{Acknowledgement}

The work of M.H and S.A. has been supported by the Austrian Science Fund (FWF), project P 30747-N32.

\end{document}